\documentclass[times]{nlaauth}
\usepackage{moreverb}

\def\volumeyear{2016}
\usepackage{relsize}
\usepackage{cite}
\usepackage{subfig}
\newcommand{\coloneq}{\mathrel{\mathop:}=}

\newcommand{\vect}[1]{\boldsymbol{#1}}
\newcommand{\tvect}[1]{{\bf #1}}
\newcommand{\xvect}[1]{{\it #1}}
\newcommand{\matr}[1]{\mathlarger{\mathbf{#1}}}
\newcommand{\tmatr}[1]{\mathlarger{\mathcal{#1}}}

\newcommand{\R}{\ensuremath{\mathbb{R}}}

\newcommand{\N}{\ensuremath{\mathbb{N}}}

\newcommand{\dt}{\Delta t}
\newcommand{\norm}[1]{\left\lVert#1\right\rVert}

\theoremstyle{plain}
\newtheorem{definition}{Definition}
\newtheorem{lemma}{Lemma}
\newtheorem{theorem}{Theorem}
\newtheorem{remark}{Remark}

\renewcommand{\proofname}{Proof}

\begin{document}

%\runningheads{A.~N.~Other}{A demonstration of the \journalabb\
%class file}

\title{A multigrid perspective on the parallel full approximation scheme in space and time}
\author{Matthias Bolten\affil{1}, Dieter Moser\affil{2}\corrauth, Robert Speck\affil{2}  }

\address{
\affilnum{1}Department of Mathematics, Universit\"at Kassel, Germany. \break
\affilnum{2} J\"ulich Supercomputing Centre, Forschungszentrum J\"ulich GmbH, Germany.    
}

\corraddr{E-mail: d.moser@fz-juelich.de}

\begin{abstract}
  For the numerical solution of time-dependent partial differential equations, 
time-parallel methods have recently shown to provide a
promising way to extend prevailing strong-scaling limits of numerical
codes. One of the most complex methods in this field is the ``Parallel
Full Approximation Scheme in Space and Time'' (PFASST). 
PFASST already shows promising results for many use cases and many more is work in progress. 
However, a solid and reliable mathematical foundation is still missing. 
We show that under certain assumptions the PFASST algorithm can be conveniently and
rigorously described as a multigrid-in-time method.
Following this equivalence, first steps towards a comprehensive analysis of PFASST using block-wise local Fourier analysis are taken. 
The theoretical results are applied to examples of diffusive and advective type.

\end{abstract}


\maketitle

\vspace{-6pt}

\section{Introduction}
Due to the rapid increase of the number of cores of todays and future HPC systems the demand for new parallelization strategies has grown rapidly in the last decades.
When the speedup of a parallelization of the spatial dimensions is saturated, one general idea is to utilize parallelization of the temporal dimension.
In \cite{Burrage1997} we find a classification of such methods, divided into 
parallelization \textit{across the step, across the method} or \textit{across the problem}. 

Direct time-parallel methods mostly belong to the class of parallelization across the method, 
examples are certain parallel Runge-Kutta methods \cite{iserles1990theory,
butcher1997order}. 
Only modest parallel speedup is expected for these methods, 
because the number of processing units used for the parallelization are bound by e.g. the number of Runge-Kutta stage values.
Other direct methods for parallel-in-time integration include RIDC~\cite{ChristliebEtAl2010}, ParaExp~\cite{Guettel2013parallel}, tensor-product space-time solvers~\cite{MadayRonquist2008} or methods using Laplace transformation~\cite{Sheen2003parallel}.

If a method decomposes the problem into subproblems which are solvable in a parallel manner and couples these subproblems using an iterative method, 
it typically belongs to the class of parallelizations across the problem.
The most prominent example are waveform relaxation methods~\cite{gander1999waveform, vandewalle1989parallel}, which are part of the broad area of domain decomposition methods.

First ideas of parallel-in-time integration date back to Nievergelt in 1964~\cite{Nievergelt1964}, which belongs to the class of multiple shooting methods and hence to the class of parallelizations across the step. 
More parallel-in-time integration methods were found in the area of multiple-shooting methods~\cite{ChartierPhilippe1993, BellenZennaro1989}. 
Among them, in 2001 by Lions et al., \textit{Parareal}~\cite{LionsEtAl2001}
renewed the interest in parallel-in-time methods and sparked many new papers in its field.
The success of Parareal is accounted to its simplicity and applicability:
Only a fine but expensive and a coarse but cheap propagator in time have to be provided by the user.
Then, parallelization across the temporal dimension can be achieved in an iterative prediction-correction manner. 
In principle, the number of processing units is not bounded, but depends on the actual decomposition of the time domain.

Parareal influenced other methods (see \cite{GanderEtAl2013_DDM})
or even inspired the design of new methods.
In~\cite{Minion2010}, the Parareal approach is coupled to iterative solvers of a collocation problem, the so called spectral deferred correction (SDC) methods. 
This approach is extended to the ``parallel full approximation scheme in space and time'' (PFASST) in~\cite{EmmettMinion2012}.
PFASST adopts and evolves the characteristics of Parareal by interweaving its iterations with those of the local SDC scheme. 
In addition, PFASST uses ideas from the theory of nonlinear FAS multigrid methods. 

Multigrid methods in general have a long-standing successful history and a solid mathematical basis, see e.g.~\cite{trottenberg2000multigrid}. 
Regarding parallel-in-time integration, the first attempt using multigrid ideas dates back to Hackbusch in 1984~\cite{Hackbusch1984}. 
Since then, multigrid methods were further developed and resulted, e.g., in the multigrid waveform relaxation \cite{LubichOstermann1987, Vandewalle1994space}, in
multigrid reduction-in-time \cite{Falgout2014parallel}, or in classical space-time-multigrid \cite{HortonVandewalle1995, Gander2014analysis}.
%
% - hier koennte ein überblick über die strategien mit denen man versuch hat mg methoden für parallel-in-time fit zu machen.
%
All these classes are not strictly separated from each other.
Oftentimes methods may be reformulated to fit into a new class.
A prominent example is Parareal itself: 
it was reformulated as a multiple shooting method as well as a multigrid method in~\cite{GanderVandewalle2007_SISC}, which in turn paved the way for a comprehensive analysis of Parareal.

This already shows the growing number and diversity of parallel-in-time methods.
A classification of PFASST into the diversity of methods contributes to the understanding of PFASST by opening up the opportunity to use different mathematical tools from different fields.
In particular, multigrid theory offers a variety of tools such as local Fourier analysis to estimate the convergence and to obtain a priori error bounds. 
A mathematical analysis becomes more and more important for the comparison of these algorithms and the design of algorithms for different applications.

The goal of this work is to associate PFASST with multigrid methods and apply the tools we find in multigrid theory along the lines of two standard problems, namely diffusion and advection. 
This sheds light on a general strategy how to estimate the convergence rate of PFASST and hereby the number of iterations needed to achieve a certain precision.

To achieve this goal, we proceed as follows: In Section~\ref{ch:preliminaries}, we we introduce the notation and preliminaries necessary to state a matrix formulation of PFASST and its constituents. 
In particular, we introduce the collocation problem and the notation to deal with the nested multilevel structure of our setting.
On this basis, we introduce the spectral deferred correction and its multi-level enhancement in matrix form in Section \ref{ch:sdc} and \ref{ch:mlsdc}. 
Then, we introduce PFASST in algorithmic form in Section~\ref{ch:pfasst_algorithm} which is then converted into matrix form in \ref{sec:block_gauss_seidel} to \ref{ch:assembling_pfasst}.  
This matrix form facilitates the use of ideas from multigrid analysis in Section~\ref{ch:lfa_and_transformation_matrices} to \ref{ch:periodicity_in_time}, which leads to a block-decomposition of the iteration matrix of PFASST.
In Section~\ref{ch:numerical_experiments}, we introduce four strategies to estimate the convergence rate of PFASST. 
The work is closed with an outlook and a conclusion in Section~\ref{ch:outro}.

\section{The Parallel Full Approximation Scheme in Space and Time}
\label{ch:pfasst}
We start with a brief introduction of the building blocks of PFASST from the perspective of linear iterative solvers. 
To this end, we restrict ourselves to linear autonomous ordinary differential equations and---for the multi-level parts---to two levels only.
We will comment on these restrictions in Section~\ref{ch:outro}.
% This is made possible by reducing the algorithm's to two-grid versions and focus on linear autonomous ordinary differential equations.

\subsection{Preliminaries and Notation}
\label{ch:preliminaries}

The starting point is the linear autonomous ordinary differential equation in the Picard formulation
\begin{align}
  \xvect{U}(t) = \xvect{U}_0 + \int_{t_0}^t \matr{A}\xvect{U}(\tau) \mathrm{d}\tau, \quad  t \in \left[ t_0, T \right],
\label{eq:ode_integral_form}
\end{align}
where $\matr{A}$ is a discretized spatial operator, e.g., stemming from a method of line discretization of a partial differential equation.
For the discretization in the temporal dimension the time domain $\left[ t_0, T \right]$ is divided into $L$ subintervals. 
Each subinterval $\left[ t_{l-1},t_{l} \right]$, with $l \in \left\{ 1,\ldots,L \right\}$, contains a set of $M$ nodes $\left\{ \tau_1, \ldots, \tau_M\right\}$.
We choose
  \begin{align}
    \begin{split}
   	0 = t_0<\ldots<t_L=T,&\quad t_l<\tau_1<\ldots<\tau_M = t_{l+1},\\
   	\dt = t_{l+1}-t_l,&\quad \Delta\tau_m = \tau_{m+1}-\tau_m.
    \end{split}
  \label{eq:time_subintervals}
  \end{align}
Each set of nodes $\left\{ \tau_1, \ldots, \tau_M\right\}$ are used as quadrature nodes for the numerical integration with rules like, e.g., Gau\ss-Radau or Gau\ss-Lobatto.
% Since the last quadrature node coincides with the right border of the particular subinterval, the use of the quadrature rule yields
% the numerical approximation of the sought value at the end of each subinterval. 
Note that the last quadrature node coincides with the right border of the particular subinterval, which simplifies the formal notation of the algorithm.
The results translate to other quadrature rules with minor modifications, though.
%With spatial paralellization, usually a set of processors treats the spatial problem, 
%and we assign such a set to each subinterval. 
Furthermore, if a mathematical entity like a set of numerical values or a certain matrix
belongs to a subinterval $\left[t_l, t_{l+1}\right]$ we denote it e.g.~by $\vect{U}_{\left[ t_l, t_{l+1}\right]}$ (if it is not clear from the context). 

Due to the nested structure and the distinct treatment of spatial and temporal dimensions, an appropriate notation is needed. 
Continuous functions are always represented by lower case letters, discretized and semi-discretized functions are the upper case version. 
Let $u(t,x)$ be a function in space and time, defined on the domain $\left[ t_0,T \right] \times \R$, with $T \in \R_+$. 
For $N$ degrees-of-freedom in space $x_1, ..., x_N$ we use the notation
$$\xvect{U}(t) = \left( u(t,x_1),u(t,x_2),\dots,u(t,x_N) \right)^T \in \R^{N}, \quad t\in \left[ t_0,T \right]$$
for semi-discretization in space.
A full space-time discretization is denoted as
\begin{align*}
  \vect{U}_{\left[ t_{l-1}, t_{l}\right]} &= \left( \xvect{U}(\tau_1),\xvect{U}(\tau_2),\dots,\xvect{U}(\tau_M) \right)^T \in \R^{M \cdot N}, \quad  \tau_i \in[t_{l-1},t_l],\ l\in\left\{ 1,\ldots,L \right\},\\
  \vect{U}   &= \left( \vect{U}_{\left[ t_{0}, t_{1}\right]} , \ldots , \vect{U}_{\left[ t_{L-1}, T\right]} , \right)^T  \in \R^{M \cdot N \cdot L}. 
\end{align*}

On each subinterval a \textit{collocation problem} is posed. 
It arises, when quadrature is used as a numerical counterpart to the integration in \eqref{eq:ode_integral_form}.
The basis for most quadrature formulations is the interpolation, 
easily expressed using the Lagrange polynomial basis $\left\{ \ell_i \right\}_{i=1}^{M}$,
with
\begin{align}
  \ell_i(s) \coloneq \prod_{k=1,j\neq i}^M\frac{s-\tau_k}{\tau_i-\tau_k}.
  \label{eq:lagrange_polynomials}
\end{align} 
If we weight each Lagrange polynomial with the evaluation of the function $f(t)$ at the point $\tau_i$ and sum them up, 
we get the interpolation polynomial of the function $f(t)$, which is exact on the nodes $\left\{ \tau_1,\ldots \tau_M \right\}$. 
Now, quadrature is nothing more than using the exact integration values of the interpolation polynomial as approximations for the integration of $f(t)$. 
The following definition employs this strategy. 

\begin{definition}
  Let  $a < \tau_1<\tau_2<\ldots<\tau_M = b$ be the set of quadrature nodes and $\matr{Q}$ the quadrature matrix with entries
  $$ q_{i,j} = \int_{a}^{\tau_j} \ell_i(\tau) \,\mathrm{d} \tau,\quad i,j = 1,...,M. $$
  We discretize \eqref{eq:ode_integral_form} 
  at the quadrature nodes, 
  using the matrix $\matr{Q}$ as approximation of the integral 
  and obtain this set of linear equations:
  \begin{align*}
    \xvect{U}(\tau_i) = \xvect{U}(t_0) + \sum_{j=1}^{M} q_{i,j} \matr{A} \xvect{U}\left(\tau_i\right),\quad i = 1,...,M. 
  \end{align*}
  Using the  Kronecker product and the vector of ones $\matr{1}_{M} \in \R^M$ we write this system of linear equations as
  \begin{align*}
    \vect{U} = \vect{U}_0 + \Delta t \matr{Q}\otimes\matr{A} \vect{U}, \quad \mbox{with} \; \vect{U}_0 = \matr{1}_{M} \otimes \xvect{U}(t_0),
  \end{align*}
  or, more compactly,
  \begin{align}
    \matr{M}\vect{U}=\left( \matr{I} - \Delta t \matr{Q} \otimes \matr{A} \right)\vect{U} = \vect{U}_{0}.
    \label{eq:colloc_problem}
  \end{align}
  This problem is called \textbf{collocation problem} on $[a,b]$.
  \label{def:colloc_problem}
\end{definition}

The set of quadrature nodes determines the kind of quadrature.
Well-known quadrature rules are Chebyshev, Gau\ss-Legendre, Gau\ss-Radau, and Gau\ss-Lobatto. 
These quadrature rules have a spectral order, 
which is reflected in the high order of the numerical solution of the collocation problem. 
Gau\ss-Radau and Gau\ss-Lobatto quadrature rules use quadrature nodes 
which are in accordance with \eqref{eq:time_subintervals}. 
Due of the higher order we focus on the Gau\ss-Radau quadrature rule in this paper.

Finally, the PFASST algorithm is working on a hierarchy of discretizations.
As mentioned before, we focus on the two-level version with spatial coarsening only, i.e.~PFASST is solving on a coarse and a fine level in space.
For both levels a separate set of operators and value vectors is needed.
The coarse level versions are simply denoted with a tilde, e.g. $\matr{\tilde A}$ is the coarse level version of $\matr{A}$.

\subsection{Spectral Deferred Corrections}
\label{ch:sdc}

Instead of directly solving the collocation problem on a subinterval, the spectral deferred corrections method (\textbf{SDC}) utilizes a low-order method to generate an iterative solution that converges to the collocation solution $\vect{U}$. 
SDC was first introduced by Dutt et al. \cite{DuttEtAl2000} as improvement of deferred correction methods \cite{FrankUeberhuber1977}. 
In the last decade, SDC was accelerated with GMRES or other Krylov subspace methods \cite{HuangEtAl2006}, enhanced to a high-order splitting method \cite{LaytonMinion2004,Minion2004,BourliouxEtAl2003}, 
and found its way into the domain of parallel respectively time-parallel computing \cite{GuibertDervout2007,MinionEtAl2008}, in particular within PFASST~\cite{Minion2010,EmmettMinion2012}.

Regarding the setting of this paper, we cast SDC as a \textit{preconditioned Richardson iteration method} for the collocation problem as defined in Definition~\ref{def:colloc_problem}.
This was pointed out earlier by various authors. For example in the work of Weiser et al. \cite{Weiser2013} this interpretation was used to optimize the convergence speed of SDC.

% Generally SDC posses the flexibility to work as a IMEX scheme. For the ease of notation, we pass on the flexibility and focus on SDC as a full implicit solver. 

A general preconditioned Richardson iteration, noted as
\begin{align}
   \vect{U}^{k+1}   &= \vect{U}^k + \matr{P}^{-1}(\vect{c} - \matr{M}\vect{U}^k),
   \label{eq:preconditioned_richardson}
\end{align}
is fully described by the preconditioner $\matr{P}$, the system matrix $\matr{M}$, and the right-hand side $\vect{c}$ of the linear equation under consideration.
$\matr{P}$ has to be easy to invert, while being an accurate alternative for the system matrix $\matr{M}$.
The SDC method follows this approach by replacing the dense quadrature matrix $\matr{Q}$ by a lower triangular matrix $\matr{Q}_{\Delta}$.  
One simple way to generate a lower triangular matrix is to use the rectangle rule for quadrature instead of the Gau\ss-Radau rule. 
In \cite{Weiser2013} an LU decomposition of $\matr{Q}$ provides a $\matr{Q}_{\Delta}$ which results in better convergence properties than the use of the simple rectangle rule while requiring the same computational effort.

The particular choice
\begin{align}
  \matr{P}_{\mathrm{SDC}} = \matr{I} - \dt\matr{Q}_\Delta \otimes \matr{A}, \quad \mbox{and} \quad 
  \vect{c} = \left(\xvect{U}(t_0),\xvect{U}(t_0),\ldots,\xvect{U}(t_0) \right)^T\in\mathbb{R}^{NM},
  \label{eq:sdc_preconditioning_matrix}
\end{align}
then allows us to write SDC as preconditioned Richardson iteration with system matrix $\matr{M}$ as defined in Def.~\ref{def:colloc_problem} and where the right-hand side is given by the initial values $\xvect{U}(t_0)$ of the ODE spread on each node.
If SDC is used on another subinterval than the first, the right-hand side consists of a numerical approximation of $\xvect{U}(t_l)$ spread on each node.
In order to start the iteration an initial iteration vector $\vect{U}^{0}$ is needed. 
For SDC, the right-hand side is an apparent choice for a initial iteration vector. 
%One could use a simple method like an explicit euler method to compute an approximation $\xvect{\hat U}(\tau_j)$ on each node, stacked together it makes a great initial iteration vector.
With these choices, one Richardson iteration is equivalent to one SDC sweep~\cite{WinkelEtAl2014,Weiser2013}. 
The iteration matrix of SDC is simply given by
\begin{align}
  \begin{split}
    \matr{T}_{\mathrm{SDC}}&= \matr{I} - \matr{P}^{-1}_{\mathrm{SDC}} \matr{M} \\
    &= \matr{I} - \left( \matr{I} - \Delta t \matr{Q}_{\Delta} \otimes \matr{A} \right)^{-1}\left( \matr{I} - \Delta t \matr{Q} \otimes \matr{A}  \right),
  \end{split}
  \label{eq:it_matrix}
\end{align}

Note that if we just use the lower triangular part of the $\matr{Q}$ matrix as $\matr{Q}_{\Delta}$,
the preconditioned Richardson iteration is a Gau\ss-Seidel iteration.
With $\matr{Q}_{\Delta}$ being a simpler integration rule or stemming from the LU decomposition of $\matr{Q}$ instead of the lower triangular part of $\matr{Q}$, we characterize SDC as \textbf{approximative Gau\ss-Seidel iteration}.

\subsection{Multi-level Spectral Deferred Corrections}
\label{ch:mlsdc}
The next step towards PFASST is the introduction of multiple levels in space (and time, which we will not consider here).
This leads to a multi-level spectral deferred corrections method called (\textbf{MLSDC}), first introduced and studied in \cite{SpeckEtAl2014_BIT}. 
Here, SDC iterations (called ``sweeps'' in this context) are performed alternately on a fine and on a coarse level in order to shift work load to coarser, i.e.~cheaper, levels.
These cheaper levels are obtained, e.g., by reducing the degrees-of-freedom in space or the order of the quadrature rule in time.
Therefore, MLSDC requires  suitable interpolation and restriction operators $\matr{T}_C^F$ and $\matr{T}_F^C$, and a coarse-grid correction in order to transfer information between the different levels.
As a consequence, MLSDC can be written as a FAS-multigrid-like iteration.
Like SDC it solves the collocation problem in an iterative manner, using the same initial iteration vector.
For our purpose we derive a two-level version from \cite{SpeckEtAl2014_BIT} as:

\begin{enumerate}
  \item Perform $n_F$ fine SDC sweep using the values $\vect{U}^{k}$ according to \eqref{eq:preconditioned_richardson}.
    This yields provisional updated values $\vect{U}^{*}$.
\item Sweep from fine to coarse:
  \begin{enumerate}
  \item Restrict the fine values $\vect{U}^{*}$ to the coarse
    values $\tilde{\vect{U}}^{k}$.
  \item Compute the FAS correction 
    $\vect{\tau}^{k+1} = \matr{\tilde M}\vect{\tilde U}^{k} - \matr{T}_C^F \matr{M} \vect{U}^{*}$% using $\vect{U}^{k+1}$ and $\tilde{\vect{U}}^{k}$
  \item Perform $n_C$ coarse SDC sweeps
    beginning with $\tilde{\vect{U}}^k$ and the FAS correction
    $\vect{\tau}^k$.  This yields new values $\tilde{\vect{U}}^{k+1}$
  \end{enumerate}
\item Sweep from coarse to fine :
    Compute the interpolated coarse correction $\vect{\delta}^k$ and
    add it to $\vect{U}^{*}$ to obtain $\vect{U}^{k+1}$
\end{enumerate}

Note that we use the FAS correction strategy here to match the description of~\cite{SpeckEtAl2014_BIT}.
This is just a question of notation, 
because in the linear case using this correction strategy 
is equivalent to the standard coarse-grid correction \cite{trottenberg2000multigrid}.
Note further, that we will only perform one fine and one coarse SDC sweep in each MLSDC iteration, i.e. $n_F = n_C = 1$.
The next lemma shows that we can cast this algorithm as a preconditioned Richardson iteration, too.

\begin{lemma}
  Let $\matr{T}_C^F \in \R^{ N M \times\tilde{N}\tilde{M}}$ and $\matr{T}_F^C \in \R^{\tilde{N}\tilde{M}\times NM}$ be the prolongation and restriction operators 
  which transfer information between the coarse and fine level.
  We describe the same problem on a fine space-time grid with the system matrix $\matr{M}$ and on a coarse space-time grid with $\tilde{\matr{M}}$. 
  For both levels we use a preconditioned Richardson iteration method, which
  is characterized by $\matr{P}$ and $\tilde{\matr{P}}$ to solve $\matr{M}\vect{U}=\vect{c}$ and $\matr{\tilde{M}}\vect{\tilde{U}} = \matr{T}_F^C \vect{c} = \vect{\tilde c}$, respectively.
  Then a combination of both methods using coarse-grid correction can be written as 
  \begin{align}
    \begin{split}
      \vect{U}^{k+\frac{1}{2}} &= \vect{U}^k + \matr{T}_C^F \matr{\tilde{P}}^{-1}_{\mathrm{SDC}}\matr{T}_F^C\left( \vect{U}^0 - \matr{M} \vect{U}^k \right)\\
      \vect{U}^{k+1} &= \vect{U}^{k+\frac{1}{2}} + \matr{P}^{-1}_{\mathrm{SDC}}\left( \vect{U}^0 - \matr{M} \vect{U}^{k+\frac{1}{2}}  \right)
    \end{split}
    \label{eq:two_level_fas_corrected}
  \end{align}
  It is possible to write \eqref{eq:two_level_fas_corrected} in form of \eqref{eq:preconditioned_richardson}, 
  using a new preconditioner $\vect{P}_{\mathrm{MLSDC}}$, where
\begin{align}
  \vect{P}_\mathrm{MLSDC}^{-1} = \matr{T}_C^F \matr{\tilde{P}}^{-1}_{\mathrm{SDC}}\matr{T}_F^C+\matr{P}^{-1}_{\mathrm{SDC}}-\matr{P}^{-1}_{\mathrm{SDC}} \matr{M} \matr{T}_C^F \matr{\tilde{P}}^{-1}_{\mathrm{SDC}}\matr{T}_F^C.
  \label{eq:nested_preconditioning_matrix}
\end{align}
Following \eqref{eq:it_matrix} and using $\vect{P}_{\mbox{SDC}}$ and $\vect{\tilde P}_{\mbox{SDC}}$ yields the MLSDC iteration matrix  
\begin{align}
  \matr{T}_{\mathrm{SDC}} =  \matr{I} - \left(\matr{T}_C^F \matr{\tilde{P}}_{\mathrm{SDC}}^{-1}\matr{T}_F^C + \matr{P}_{\mathrm{SDC}}^{-1} -\matr{P}_{\mathrm{SDC}}^{-1} \matr{M} \matr{T}_C^F \matr{\tilde{P}}_{\mathrm{SDC}}^{-1}\matr{T}_F^C \right)\matr{M}.
  \label{eq:mlsdc_it_matrix}
\end{align}

\label{lem:two_level_fas_corrected}
\end{lemma}
\begin{proof}
  Let $\vect{U}^k$ be the result of the last iteration on the fine level. 
  For the proof we start in the middle of the algorithm.
  First we compute the FAS correction 
  \begin{align}\label{eq:fas_mlsdc}
    \vect{\tau}^k = \matr{\tilde{M}}\matr{T}_F^C \vect{U}^k - \vect{T}_F^C \matr{M}\vect{U}^k
  \end{align}
and use it to modify $\vect{\tilde{c}}$ for the next iteration on the coarse level. 
We start the iteration on the coarse level with
\begin{align*}
  \vect{\tilde{U}}^{k+1} 
  = \vect{\tilde{U}}^{k} + \matr{\tilde{P}}^{-1}_{\mathrm{SDC}}\left( \vect{\tilde{c}} + \vect{\tau}^k - \matr{\tilde{M}} \vect{\tilde{U}}^{k} \right)   
  = \matr{T}_F^C \vect{U}^k + \matr{\tilde{P}}^{-1}_{\mathrm{SDC}} \matr{T}_F^C  \left( \vect{c}  - \matr{M}\vect{U}^k \right),
\end{align*}
with the restricted value $\vect{\tilde{U}}^{k} = \matr{T}_F^C \vect{U}^k$.
Then, we compute the coarse correction 
$$ \vect{\delta}^k = \matr{T}_C^F\left( \vect{\tilde{U}}^{k+1} - \matr{T}_F^C \vect{U}^k \right) $$
and obtain the half-step 
\begin{align*}
    \vect{U}^{k+\frac{1}{2}} 
  = \vect{U}^k + \vect{\delta}^k 
  = \vect{U}^k + \matr{T}_C^F \matr{\tilde{P}}^{-1}\matr{T}_F^C\left( \vect{U}^0 - \matr{M} \vect{U}^k \right)
\end{align*}
after some algebraic manipulations.
Using this half-step for the next iteration on the fine level gives \eqref{eq:two_level_fas_corrected}. 
Simple algebraic manipulations, after inserting the half-step into the second step, 
yield the preconditioner \eqref{eq:nested_preconditioning_matrix}, which immediately leads to the iteration matrix \eqref{eq:mlsdc_it_matrix}.
\end{proof}

For the matrix formulation it is irrelevant whether the MLSDC step starts with the computation on the fine or the coarse level.
To comply with the literature, we leave the algorithm of MLSDC in the original order, while changing the order for the matrix formulation.
%By contrast, in Sec.~\ref{ch:pfasst_algorithm}, the order in the algorithm is changed to match the notation with the matrix formulation.

As a part of PFASST, MLSDC corresponds to the computation performed on each subinterval. 
Adding a communication framework between the MLSDC iterations performed on each subinterval 
leads to PFASST. However, adding the communication framework yields a structure  
similar to the one we have seen in Lemma~\ref{lem:two_level_fas_corrected}.
%Therefore, the form \eqref{eq:two_level_fas_corrected} is found again with adjusted operators and therefore
%the iteration matrix of PFASST has a similar form. 

\subsection{The PFASST algorithm}
\label{ch:pfasst_algorithm}

\begin{figure}[th]
	\centering	
	\includegraphics[scale=1]{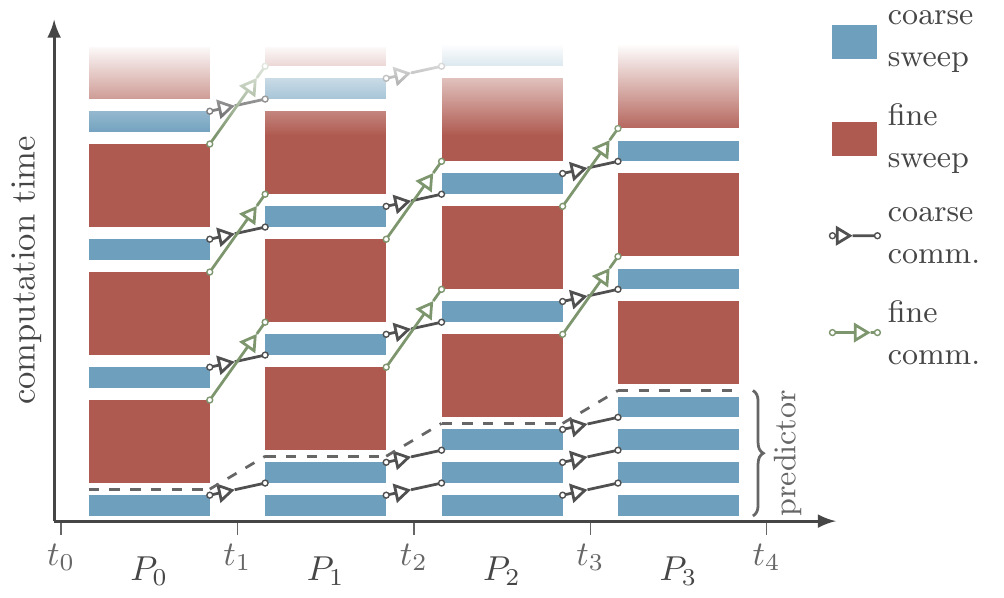}
	\caption{
	  %Schematic representation of PFASST.
	  Schematic representation of the PFASST algorithm with two levels and four processes $P_0,...,P_3$ handling four parallel time steps. 
	  %Only communication on the coarsest level is blocking, communication on other levels can be overlapped with computation. 
	  Created using \texttt{pfasst-tikz}~\cite{pfasst_tikz}.}
	\label{fig:pfasst_full}
\end{figure}

The time-parallel PFASST algorithm in its final form was introduced in \cite{EmmettMinion2012} as a 
combination of SDC methods \cite{DuttEtAl2000} with Parareal \cite{LionsEtAl2001} 
using an FAS correction strategy to allow for efficient spatial coarsening along the level hierarchy.

We explain PFASST on the basis of the schematic representation in Figure~\ref{fig:pfasst_full}. 
First of all, we see the time domain, decomposed into subintervals, on the x-axis. 
On the y-axis we see the elapsed computational time.
Each processor is assigned to a subinterval, where it performs MLSDC iterations and sends intermediate results on each level to the next processor.
The blue and red blocks represent the SDC sweeps on the coarse and fine level. 
These blocks are connected through FAS corrections to the subjacent blocks (red to blue). 
The arrows represent the communication between the processors. 
Before starting with the actual PFASST iterations, a prediction phase, represented by the first blue blocks near the x-axis, computes suitable initial values for the iterations to come. 

Based on the schematic representation and the full algorithm description in \cite{EmmettMinion2012}, we state a two-level version without the prediction phase.
Let $\xvect{U}_{[t_{l-1},t_l],m}^k$ be the value on the $l$-th subinterval at the $k$-th iteration and the $m$-th node. 
We have
$$\vect{F}_{[t_{l-1},t_l]}^k = [\matr{A}\xvect{U}_{[t_{l-1},t_l],1}^k,\ldots,\matr{A}\xvect{U}_{[t_{l-1},t_l],M_l}^k] \; \mbox{and} \; \vect{U}_{[t_{l-1},t_l]}^k = \left[ \xvect{U}_{[t_{l-1},t_l],1}^k, \ldots, \xvect{U}_{[t_{l-1},t_l],M_l}^k  \right],$$
where $M_l$ is the number of nodes on the $l$-th interval.
An upper bar, e.g. $\bar{\xvect{U}}_{l-1}^{k+1}$, 
indicates that this value was sent by the preceding processor. 
These values are used as a new right-hand side to the collocation problem on the following subinterval. 
Denote the initial values for each subinterval as $\xvect{U}_{[t_{l-1},t_l],m}^0$.
Prepared with this notations, we are ready to formulate the PFASST algorithm:
\begin{enumerate}
\item Go down to the coarse level:
  \begin{enumerate}
     \item Restrict the fine values $\vect{U}_{[t_{l-1},t_l]}^{k}$ to the coarse values $\vect{\tilde{U}}_{[t_{l-1},t_l]}^{k}$ and compute $\vect{\tilde{F}}_{[t_{l-1},t_l]}^k$.
     \item Compute FAS correction $\vect{\tau}^k$, using $\vect{\tilde{F}}_{[t_{l-1},t_l]}^k$ and $\vect{F}_{[t_{l-1},t_l]}^{k}$.
     \item If $l>0$, then receive the new initial value $\xvect{\tilde{\bar{U}}}_l^{k}$ from processor $\mathbf{P}_{l-1}$ and compute $\vect{\tilde{F}}_{[t_{l-1},t_l],0}^k$, else use the initial value of the ODE. 
     \item Perform $n_C$ SDC sweeps with values $\vect{\tilde{U}}_{[t_{l-1},t_l]}^k$, $\vect{\vect{\tilde{F}}}_{[t_{l-1},t_l]}^{k}$ and the FAS correction $\vect{\tau}^k$.
       This yields new values $\vect{\tilde{U}}_{[t_{l-1},t_l]}^{k+\frac{1}{2}}$ and $\vect{\tilde{F}}_{[t_{l-1},t_l]}^{k+\frac{1}{2}}$.
     \item Send $\xvect{\tilde{U}}_{[t_{l-1},t_l],M_l}^{k+\frac{1}{2}}$ to processor $\mathbf{P}_{l+1}$ if $l<N-1$. 
           This will be received as the new initial condition $\tilde{\bar{\xvect{U}}}_{l}^{k}$  for the solver on the coarse level. 
  \end{enumerate}
\item Return to the fine level:
\begin{enumerate}
  \item Interpolate the coarse correction $\vect{\delta}^k = \vect{\tilde{U}}_{[t_{l-1},t_l]}^{k+\frac{1}{2}} - \vect{\tilde{U}}_{[t_{l-1},t_l]}^k$ 
    and add to $\vect{U}_{[t_{l-1},t_l]}^{k}$, yielding $\vect{U}_{[t_{l-1},t_l]}^{k+\frac{1}{2}}$. 
    Recompute $\vect{F}_{[t_{l-1},t_l]}^{k+\frac{1}{2}}$.
  \item If $l>0$, then receive the new initial value $\xvect{\bar{U}}_{l-1}^{k}$ from processor $\mathbf{P}_{l-1}$, else take the initial value of the ODE. 
  \item Interpolate coarse correction vector $\xvect{\delta}^k = \tilde{\bar{\xvect{U}}}_{l-1}^{k+\frac{1}{2}} - \tilde{\bar{\xvect{U}}}_{l-1}^{k}$ and add it to  $\bar{\xvect{U}}_{l}^{k}$, 
    yielding $\bar{\xvect{U}}_{l}^{k+\frac{1}{2}}$.
    Recompute $\xvect{F}_{[t_{l-1},t_l],1}^{k+\frac{1}{2}}$.
\end{enumerate}
  \item Perform $n_F$ fine SDC sweeps using the values $\vect{U}_{[t_{l-1},t_l]}^{k+\frac{1}{2}}$ and $\vect{F}_{[t_{l-1},t_l]}^{k+\frac{1}{2}}$. 
    This yields values $\vect{U}_{[t_{l-1},t_l]}^{k+1}$ and $\vect{F}_{[t_{l-1},t_l]}^{k+1}$.
  \item Send $\xvect{U}_{[t_{l-1},t_l],M_l}^{k+1}$ to processor $\mathbf{P}_{l+1}$ if $l<N-1$. 
        This will be used as initial value $\bar{\xvect{U}}_{l+1}^{k+1}$ in the next iteration on the fine level. 
\end{enumerate}

This form of the PFASST algorithm is suitable for implementation, 
but rather not for the mathematical analysis. 
It is especially difficult to capture how the parts influence each other.
To overcome this limitation, we now change the perspective: Instead of building the algorithm in a ``vertical'' way (MLSDC on each subinterval), 
we look at all intervals at once in a ``horizontal'' way, i.e., we analyze how the different components of PFASST act on the full time-domain $[t_0,T]$.

\section{A multigrid perspective}
\label{ch:multigrid_perspective}
In this section, the perspective is shifted from solvers on one specific subinterval to the interaction of the solvers on the whole time domain $\left[ t_0,T \right]$. 
We begin with stating the composite collocation problem. 

\begin{definition}
  Let the interval $\left[ t_0,T \right]$ be decomposed as in \eqref{eq:time_subintervals} into $L$ subintervals $\left[ t_{l}, t_{l+1} \right]$. 
  On each subinterval a collocation problem in the form of \eqref{eq:colloc_problem}, 
  denoted by $\matr{M}_{[t_l,t_{l+1}]}$, is posed.
  The collocation matrix on the whole time domain is then defined as
\begin{align*}
  \matr{M}_{[t_0,T]} &= 
  \begin{pmatrix}
    \matr{M}_{[t_0,t_1]} &  & & \\
    -\matr{N}		 & \matr{M}_{[t_1,t_2]} & & \\
    			 & \ddots           & \ddots & \\
			 &  &  -\matr{N} & \matr{M}_{[t_{L-1},T]}
  \end{pmatrix}\in \R^{NML}, 
		       \; \mbox{with} \;    \\
   \matr{N} &=  
    \begin{pmatrix}
      0 & 0 & \cdots & 1 \\
      0 & 0 & \cdots & 1 \\
       \vdots & \vdots & & \vdots \\
      0 & 0 & \cdots & 1 \\
    \end{pmatrix} \otimes \matr{I}_{N} \in \R^{NM}.
\end{align*}
The operator $\matr{N}$ handles how the new starting value for the upcoming interval is produced.
Furthermore, stacking together
  \begin{align*}
    \vect{c}_{[t_l,t_{l+1}]} &=
    \begin{cases}
      \vect{U}_0, &\text{ for } l=0\\
      \vect{0},  &\text{ for } l>0
    \end{cases} \in \R^{NM},
  \end{align*}
form the righ-hand side $\vect{c}_{[t_0,T]}$ for the \textbf{composite collocation problem}
  \begin{align}
    \matr{M}_{\left[ t_0,T \right]} 
    \begin{pmatrix}
      \vect{U}_{\left[ t_0,t_{1} \right]}\\
 	\vect{U}_{\left[ t_1,t_{2} \right]}\\
	\vdots\\
	\vect{U}_{\left[ t_{L-1},T\right]}
    \end{pmatrix}
    = 
    \begin{pmatrix}
      \vect{U}_0 \\
 	0	\\
	\vdots	\\
	0	\\	
      \end{pmatrix} = \vect{c}_{[t_0,T]}.
    \label{eq:combined_colloc_problem}
  \end{align}
\label{def:comb_colloc_problem}
\end{definition}

Along with the definition, the block structure of our problem becomes evident. On the diagonal of the new collocation matrix, 
we find blocks of the size $NM$, each of them being associated with the subintervals $\left[ t_l,t_{l+1} \right]$. 
The operators on the subdiagonal deal with the communication between two adjacent subintervals.
When designing iterative solvers for the composite collocation problem, we also want to exploit this block structure.
Therefore, the next two sections are dedicated to the block versions of an approximate Jacobi and a approximate Gau\ss-Seidel iteration
and both will emerge from the interpretation of SDC as an approximate Gau\ss-Seidel iterative solver.
Later on, both methods, if correctly interlaced, will yield PFASST.
%The use of such block-solvers, was proposed in multigrid theory to deal with anisotropic discretization and Space-Time-Multigrid theory.

\subsection{Approximative Block Gau\ss -Seidel solver}
\label{sec:block_gauss_seidel}
The classical Gau\ss-Seidel solver is a splitting method, which incorporates the lower triangular part of the system matrix as preconditioner.
Obviously this strategy is possible in principle for the composite collocation problem, as defined in Definition~\ref{def:comb_colloc_problem}, but this would neglect the particular block structure of the problem.
Therefore, we now construct a block version of the SDC iteration, following its description as an approximate Gau\ss-Seidel solver.

Assume we perform one SDC sweep on each subinterval via
\begin{align}
  \vect{U}^{k+1}_{\left[ t_l,t_{l+1} \right]} &= \vect{U}^k_{\left[ t_l,t_{l+1} \right]} 
  + \matr{P}_{\left[t_l,t_{l+1}\right]}^{-1} \left( \vect{c}^{k+1}_{\left[t_l,t_{l+1}\right]} 
	- \matr{M}_{\left[ t_l,t_{l+1} \right]}\vect{U}^k_{\left[ t_l,t_{l+1} \right]} \right),
  \label{eq:it_ms_step_one_subinterval}
\end{align}
where $\matr{P}_{\left[t_l,t_{l+1}\right]}$ denotes the SDC preconditioner \eqref{eq:sdc_preconditioning_matrix}, 
and $\vect{c}^k_{\left[t_l,t_{l+1}\right]}$ is the right-hand side on the $l$-th subinterval in the $k$-th iteration.
In order to pass the last value forward in time to the next subinterval, we can use the matrix $\matr{N}$. 
Therefore, the right-hand side of the collocation problem can be written as 
\begin{align}
\begin{split}
  \vect{c}^k_{\left[t_0,t_{1}\right]}   &= \left[ \xvect{U}_0,\ldots, \xvect{U}_0   \right], \quad\mbox{for}\quad l=0 \;\mbox{and}\; k>0 \\
  \vect{c}^k_{\left[t_l,t_{l+1}\right]} &= \left[ \bar{\xvect{U}}^{k}_l,\ldots, \bar{\xvect{U}}^{k}_l  \right]= \matr{N}\vect{U}^k_{\left[ t_{l-1},t_l \right]} \quad\mbox{for}\quad l>0 \;\mbox{and}\; k>1\\
  \label{eq:it_ms_step_pass_information}
\end{split}
\end{align}
For some initial iteration vector $\vect{U}^{0}_{\left[ t_l, t_{l+1} \right]}$, stemming, e.g., from copying the initial value on each node of each subinterval (``spreading''), we can write this process compactly as single approximate Gau\ss-Seidel step over the whole time domain.

\begin{lemma}
  Let $\matr{M}_{\left[ t_0,T \right]}$ be the matrix of a composite collocation problem. 
  Using \eqref{eq:it_ms_step_one_subinterval} on each subinterval and passing the results via \eqref{eq:it_ms_step_pass_information}, corresponds to 
  \begin{align}
    \vect{U}^{k+1}_{[t_0,T]} =\vect{U}^k_{[t_0,T]}  + \matr{P}^{-1}_{\left[ t_0,T \right]} \left( \vect{c}_{\left[ t_0,T \right]} - \matr{M}_{\left[ t_0,T \right]} \vect{U}^k_{[t_0,T]}  \right),
  \label{eq:composite_iteration_form}
\end{align}
with
\begin{align*}
  \vect{U}^k_{[t_0,T]} =  \begin{pmatrix}
      \vect{U}_{\left[ t_0,t_{1} \right]}^k\\
 	\vect{U}_{\left[ t_1,t_{2} \right]}^k\\
	\vdots\\
	\vect{U}_{\left[ t_{L-1},T\right]}^k
    \end{pmatrix}
    \in \R^{NML}
    ,\quad \vect{c}_{\left[ t_0,T \right]} = 
    \begin{pmatrix}
      \vect{U}_0 \\
      0 \\
      \vdots\\
      0
    \end{pmatrix}
    \in \R^{NML}
\end{align*}
and
\begin{align*}
  \matr{P}_{[t_0,T]} = 
  \begin{pmatrix}
    \matr{P}_{[t_0,t_1]} &  & & \\
    -\matr{N}		 & \matr{P}_{[t_1,t_2]} & & \\
    			 & \ddots           & \ddots & \\
			 &  &  -\matr{N} & \matr{P}_{[t_{L-1},T]}
  \end{pmatrix} \in \R^{NML\times NML}.
\end{align*}
  \label{lem:it_multistep_solver}
\end{lemma}

\begin{proof}
  We multiply equation \eqref{eq:it_ms_step_one_subinterval} with $\matr{P}_{[t_l,t_{l+1}]}$ from the left
  and equation \eqref{eq:composite_iteration_form} with $\matr{P}_{[ t_0,T]}$ from the left. Comparing the resulting terms line by line reveals the equivalence.
\end{proof}

This Gauss-Seidel-like iteration can be found in Fig.~\ref{fig:pfasst_full}: Here, after each blue block which represent SDC sweeps on the coarse level, the values $\bar{\xvect{U}}_l^k$ are passed forward in time, providing new initial values for the sweep on the next interval.
Thus, the iteration on the coarse level can be identified with an approximate block Gau\ss-Seidel iteration for the composite collocation problem~\eqref{eq:combined_colloc_problem}.

%The communication between subintervals is performed in form of the operator $\matr{N}$ on the block sub-diagonal of $\matr{P}_{\left[ t_0,T \right]}$, 
%by passing the value $\bar{\xvect{U}}_l^k$ in form of a new right-hand side to the subsequent solver.
%Parts of the communication strategy of PFASST, sketched in Fig. \ref{fig:pfasst_full} as arrows between the blue blocks, are similar to the communication strategy above.

\subsection{Block Jacobi-Solver}
\label{sec:block_jacobi_solver}

The communication, emerging from the use of the approximate Block Gau\ss -Seidel solver, is blocking. 
Each processor has to wait for its predecessor. 
Hence, this is a purely serial approach.
A simple way to avoid the blocking communication is to use a approximate Block Jacobi solver, omitting the sub diagonal blocks responsible for the communication.

Assume we perform a step similar to \eqref{eq:it_ms_step_one_subinterval},
but we use the right-hand side
\begin{align}
  \begin{split}
  \vect{c}^k_{\left[t_0,t_{1}\right]}   &= \left[ \xvect{U}_0,\ldots, \xvect{U}_0   \right], \quad \mbox{for}\quad l=0 \; \mbox{and}\; k>0 \\
  \vect{c}^k_{\left[t_l,t_{l+1}\right]} &= \left[ \bar{\xvect{U}}^{k-1}_l,\ldots, \bar{\xvect{U}}^{k-1}_l  \right]= \matr{N}\vect{U}^{k-1}_{\left[ t_{l-1},t_l \right]} \quad \mbox{for}\quad l>0 \;\mbox{and}\; k>1
  \label{eq:jacobi_step_pass_information}
  \end{split}
\end{align}
instead. This means that not the result of the current but of the previous iteration of the preceding interval is used. 
%Hence, the iteration on each subinterval is not depending on the current iteration on the preceding subinterval, 
%but on the previous iteration. 
In the first iteration, the result of the prediction phase is used. 
Using the simple spreading prediction phase, this is easily achieved by choosing $\bar{\xvect{U}}^{0}_l = \xvect{U}_0$.

%As we will see in the following Lemma,
%in the matrix formulation we just have to omit the matrices $\matr{N}$ 
%on the sub diagonal of $\matr{P}_{\left[ t_0,T \right]}$ to achieve the same.

\begin{lemma}
  Let $\matr{M}_{\left[ t_0,T \right]}$ be the matrix of a composite collocation problem. 
%  and $\left\{ \matr{P}_{\left[ t_0,t_1 \right]},\ldots,\matr{P}_{\left[ t_{L-1},T \right]} \right\}$ a list of  preconditioner, associated with the subintervals. 
  Then, using \eqref{eq:it_ms_step_one_subinterval} on each subinterval and passing the results via \eqref{eq:jacobi_step_pass_information}, corresponds to 
\begin{align}
  \vect{U}^{k+1}_{[t_0,T]} =\vect{U}^k_{[ t_0,T]}  + \matr{\hat{P}}^{-1}_{\left[ t_0,T \right]} \left( \vect{c}_{\left[ t_0,T \right]} - \matr{M}_{\left[ t_0,T \right]} \vect{U}^k_{[t_0,T]}  \right)  
  \label{eq:it_bl_solver}
\end{align}
with 
\begin{align*}
  \matr{\hat{P}}_{[t_0,T]} = 
  \begin{pmatrix}
    \matr{P}_{[t_0,t_1]} &  & & \\
     		 & \matr{P}_{[t_1,t_2]} & & \\
    			 &  & \ddots & \\
			 &  &   & \matr{P}_{[t_{L-1},T]}
  \end{pmatrix},
\end{align*}
as well as $\vect{U}^k_{[t_0,T]}$ and $\vect{c}_{\left[ t_0,T \right]}$ defined as in Lemma \ref{lem:it_multistep_solver}.
\label{lem:it_bl_solver}
\end{lemma}
\begin{proof}
  Similar to the proof in Lemma \ref{lem:it_multistep_solver} a block line-wise comparison yields the equivalence.
  Especially, the influence of the sub diagonal of $\matr{M}_{[t_0,T]}$ on the communication is revealed by a block line-wise view on \eqref{eq:it_bl_solver}:
\begin{align*}
  \vect{U}^{k+1}_{[t_0,t_{1}]} &=\vect{U}^k_{[t_0,t_{1}]}  + \matr{P}^{-1}_{\left[ t_0,t_{1} \right]} 
  \left( \vect{U}_0   - \matr{M}_{\left[ t_0,t_{1} \right]} \vect{U}^k_{[t_0,t_{1}]}\right), &\mbox{for } l=0 \\ 
  \vect{U}^{k+1}_{[t_l,t_{l+1}]} &=\vect{U}^k_{[t_l,t_{l+1}]}  + \matr{P}^{-1}_{\left[ t_l,t_{l+1} \right]} 
  \left(  \matr{N}\vect{U}^k_{[t_{l-1},t_{l}]}   - \matr{M}_{\left[ t_l,t_{l+1} \right]} \vect{U}^k_{[t_l,t_{l+1}]}\right) , &\mbox{for } l > 1.
\end{align*}
The values $\matr{N}\vect{U}^k_{[t_{l-1},t_{l}]}$ are equivalent to $\matr{1}_M \otimes \bar{\xvect{U}}^{k-1}_l$.
\end{proof}

It is evident that due to the block diagonal structure of $\matr{\hat{P}}_{[t_0,T]}$ 
one block Jacobi iteration may be performed concurrently on $L$ computing units.
This approach corresponds to the sweeps on the fine (red) blocks in Fig.~\ref{fig:pfasst_full}: these sweeps can be performed in parallel, since they do not depend on the previous subinterval at the same iteration.
%Schematically this communication of the approximate block Jacobi solver is also found in PFASST and 
%is sketched in Fig. \ref{fig:pfasst_full}, as arrows connecting the red blocks. 
Therefore, the iteration on the fine level can be identified with an approximate block Jacobi iteration for the composite collocation problem~\eqref{eq:combined_colloc_problem}.

\subsection{Assembling PFASST}
\label{ch:assembling_pfasst}
Already in Section \ref{ch:mlsdc} multigrid elements where introduced to SDC to form MLSDC. 
The same ideas apply when we now interlace both iterative block solvers from above.
In order to achieve more parallelism, we compute the approximate Gau\ss-Seidel iteration step on the coarse level and the approximate
block Jacobi iteration step on the fine level, so that the more cost intensive work is done in parallel.
As the following Theorem shows, it is now possible to write PFASST in the form of \eqref{eq:two_level_fas_corrected} 
and we are able to state a iteration matrix.

\begin{theorem}
  Let $\matr{T}_F^C$ and $\matr{T}_C^F$ be block-wise defined transfer operators, which treat the subintervals independently from each other,
  let $\left\{ \matr{P}_{\left[ t_0,t_1 \right]},\ldots,\matr{P}_{\left[ t_{L-1},T \right]} \right\}$ and $\left\{ \matr{\tilde{P}}_{\left[ t_0,t_1 \right]},\ldots,\matr{\tilde{P}}_{\left[ t_{L-1},T \right]} \right\}$ be sets of 
  preconditioner for the fine and coarse level, respectively, describing SDC sweeps on $\left[ t_l, t_{l+1} \right]$ for $l\in \left\{ 0,\ldots, L-1 \right\}$ and $t_L = T$.
  Let $\matr{M}_{\left[ t_0,T \right]}$ be the composite collocation matrix of Definition \ref{def:comb_colloc_problem} and $\matr{N}$, $\matr{\tilde{N}}$ 
  be the operations to compute the initial value for the 
  following subinterval. 
  Then the linear two-level version of PFASST can be written in matrix form as 
  \begin{align}
    \begin{split}
    \vect{U}^{k+\frac{1}{2}}_{[ t_0,T]} &= 
    \vect{U}^{k}_{[t_0,T]} + 
    \matr{T}_C^F \matr{\tilde{P}}_{[t_0,T]}^{-1} \matr{T}_F^C \left( \vect{c}_{[t_0,T]} - \matr{M}_{[t_0,T]} \vect{U}^{k}_{[t_0,T]} \right)\\
   \vect{U}^{k+1}_{[t_0,T]} 
    &= 
    \vect{U}^{k}_{[t_0,T]} + 
    \matr{\hat{P}}_{[t_0,T]}^{-1} \left( \vect{c}_{[t_0,T]}  - \matr{M}_{[t_0,T]}\vect{U}^{k+\frac{1}{2}}_{[t_0,T]} \right),
    \end{split}
    \label{eq:pfasst_in_matrix_form}
  \end{align}
  with $\matr{\tilde{P}}_{[t_0,T]}$, as in Lemma \ref{lem:it_multistep_solver}, and $\matr{\hat{P}}_{[t_0,T]}$, as in Lemma \ref{lem:it_bl_solver}.
  In addition, let $\matr{N},\matr{\tilde{N}}$, such that $\matr{\tilde{N}}\matr{T}_F^C  = \matr{T}_F^C \matr{N}$ and  
  $\vect{c}_{\left[ t_0,T \right]}= [\vect{U}^0, 0 ,\ldots , 0 ]$ as well as $\vect{\tilde{c}}_{\left[ t_0,T \right]} = \matr{T}_F^C \vect{c}_{\left[ t_0,T \right]}$. 
  Finally, following \eqref{eq:it_matrix}, the PFASST iteration matrix is given by 
  \begin{align}
    \matr{T}_{\mathrm{PFASST}} = \left( \matr{I} - \matr{\hat{P}}_{[t_0,T]}^{-1}\matr{M}_{[t_0,T]} \right) \left( \matr{I} - \matr{T}_C^F \matr{\tilde{P}}^{-1} \matr{T}_F^C \matr{M}_{[t_0,T]}\right).
    \label{eq:iteration_matrix_pfasst}
  \end{align}
  \label{th:pfasst_in_matrix_form}
\end{theorem}
\begin{proof}
%The idea of this proof is to find a part in the computation, which is carried out by \eqref{eq:pfasst_in_matrix_form}, for each step describing PFASST. 
We compare systematically each step of PFASST with the sub-computations found in equation \eqref{eq:pfasst_in_matrix_form}, which expands into
\begin{align}
  \vect{\tau}_{[t_0,T]}^k &=\matr{\tilde{M}}_{[t_0,T]} \matr{T}_F^C \vect{U}^k_{[t_0,T]} - \matr{T}_F^C \matr{M}_{[t_0,T]} \vect{U}^k_{[t_0,T]} \label{eq:expanded_pfasst_1} \\
  \vect{\tilde{U}}^{k+1}_{[t_0,T]} &= \vect{\tilde{U}}^{k}_{[t_0,T]}  + \matr{\tilde{P}}^{-1}\left( \vect{\tilde{c}}_{[ t_0,T]} + \vect{\tau}_{[ t_0,T]}^k  - \matr{\tilde{M}}_{[ t_0,T]} \vect{\tilde{U}}^{k}_{[ t_0,T]}\right)\label{eq:expanded_pfasst_2}\\
   \vect{U}^{k+\frac{1}{2}}_{[ t_0,T]} &= \vect{U}^{k}_{[ t_0,T]} + \matr{T}_C^F \left( \vect{\tilde{U}}^{k+1}_{[ t_0,T]} -\matr{T}_F^C \vect{U}^k_{[ t_0,T]} \right)\label{eq:expanded_pfasst_3}\\
   \vect{U}^{k+1}_{[ t_0,T]} &= \vect{U}^{k+\frac{1}{2}}_{[ t_0,T]} + \matr{\hat{P}}^{-1}\left( \vect{c}_{[ t_0,T]} - \matr{M}_{[ t_0,T]} \vect{U}^{k+\frac{1}{2}}_{[ t_0,T]}  \right). \label{eq:expanded_pfasst_4}
\end{align}
From top to bottom, we have the computation of the FAS correction $\vect{\tau}^k$, the SDC sweep on the coarse level, coarse-grid correction, and the SDC sweep on the fine level.
PFASST's communication between the subintervals has been already derived in Lemma \ref{lem:it_multistep_solver} and Lemma \ref{lem:it_bl_solver}. 
The evaluations of right-hand side in the form of $\vect{F}$ and $\vect{\tilde{F}}$ are included in the matrix vector multiplication with $\matr{M}_{[ t_0,T]}$ and $\tilde{\matr{M}}_{[ t_0,T]}$, respectively.

The computation of the FAS correction $\vect{\tau}^k_{[ t_0,T]}$ as in~\eqref{eq:expanded_pfasst_2} differs from the formula~\eqref{eq:fas_mlsdc}, which we derived for MLSDC, i.e.~which is formed for each subinterval.
The FAS correction vector of~\eqref{eq:expanded_pfasst_2}, has additional terms:
\begin{align}
  \matr{L} =\matr{T}_F^C \matr{N} - \matr{\tilde{N}} \matr{T}_F^C
\end{align} 
with
\begin{align}
     \tau_{[ t_0,T]} = \left(\tau_{\left[ t_0,t_1 \right]},\ \tau_{\left[ t_1,t_2 \right]} + \matr{L}\vect{U}^k_{\left[ t_0,t_1 \right]},\ ...,\ \tau_{\left[ t_{N-1},T \right]} + \matr{L}\vect{U}^k_{\left[ t_{N-1},T \right]} \right)^T
    \label{eq:tau_correction_on_whole_domain}
\end{align}
However, by requirement we have $\matr{L}=\matr{0}$ and in Remark~\ref{rm:requirement} we will investigate how this requirement is met.
The iteration matrix is the result of simple algebraic manipulations.
\end{proof}
In contrast to Lemma~\ref{lem:two_level_fas_corrected} for MLSDC, we now have an additional requirement.
\begin{remark}
  Let $t_{i,j}$ be the $j$-th entry of the $i$-th row of $\matr{T}_F^C$.
  Due to the assumptions above, $\matr{L}=\matr{0}$ translate to
\begin{align*}
%  \matr{\tilde{N}}\matr{T}_{F}^{C} &= \matr{T}_F^{C} \matr{N}\\
\begin{pmatrix}
  t_{\tilde M, 1} & \cdots & t_{\tilde M, M-1} & t_{\tilde M, M} \\
  \vdots & & \vdots & \vdots\\
  t_{\tilde M, 1} & \cdots & t_{\tilde M, M-1} & t_{\tilde M, M} 
\end{pmatrix}
&= 
\begin{pmatrix}
0   & \cdots & 0 & \sum_{j=1}^M t_{1, j} \\
  \vdots &   & \vdots  & \vdots \\
0   & \cdots & 0 & \sum_{j=1}^M t_{\tilde M, j} 
\end{pmatrix}.
\end{align*}
Hence, we require that
\begin{align*}
  t_{\tilde M, j} = 0 \quad \forall \; j \in \left\{1,\ldots, M-1\right\} \qquad \mbox{and} \qquad
  t_{\tilde M, M} = \sum_{j=1}^M t_{i, j} \quad \forall \; i\in \left\{ 1,\ldots, \tilde M \right\}.
\end{align*}
If the restriction $\matr{T}_F^C$ of a constant vector yields a constant vector with the same values but a smaller dimension,
we infer that,
$$\sum_{j=1}^M t_{i, j} = 1\quad \forall \; i\in \left\{ 1,\tilde M \right\}$$
and hence $t_{\tilde M, M}=1$.
This requirement is met, when the restriction just projects the last node of the fine level onto the last node of the coarse level. 
It holds e.g. for the simple linear restriction or just injection, as long as the quadrature nodes $\tilde \tau_{\tilde M}$ and $\tau_M$ overlap for each subinterval.
\label{rm:requirement}
\end{remark}

The hierarchy of discretization on which PFASST is working and the exchange of information between those levels using FAS and coarse-grid correction obviously indicates a strong similarity to classical multigrid methods.
This relation is in particular emphasized by the iteration matrix. 
Standard multigrid methods are typically described and analyzed by their iteration matrix $\matr{T}_{\mathrm{MG}}$, which reads
\begin{align}
  \matr{T}_{\mathrm{MG}}(\nu,\mu) = 
  \left( \matr{I} - \matr{P}^{-1}_{\mathrm{post}} \matr{M} \right)^{\nu}  
  \left( \matr{I} - \matr{T}_C^F \matr{\tilde{M}}^{-1}\matr{T}_F^C \matr{M}\right)
  \left( \matr{I} - \matr{P}^{-1}_{\mathrm{pre}} \matr{M} \right)^{\mu}, 
  \label{eq:it_matrix_mg}
\end{align}
for $\nu$ post- and $\mu$ pre-smoothing steps. 
The expression in the middle is the coarse grid correction. 
In a standard two-grid algorithm, the exact solution $\matr{\tilde{M}}^{-1}$ is used at the coarse level.
In practice it is also legitimate to use the approximate solution in form of $\matr{\tilde{P}}^{-1}$.
PFASST does exactly this. 
Under the conditions of Theorem~\ref{th:pfasst_in_matrix_form}, the comparison of \eqref{eq:it_matrix_mg} and \eqref{eq:iteration_matrix_pfasst} yields that PFASST can be readily interpreted as multigrid algorithm with one post-smoothing iteration and no pre-smoothing steps.
We point out that this does not prove that PFASST actually behaves like a multigrid method in terms of convergence and robustness. 
In particular, properties like smoothing and approximation property are not necessarily satisfied and the analysis of the algorithm in this respect is left for future work.
%Despite this theoretical equivalence, the standard implementation of PFASST 
%differs from usual multigrid implementations.  
However, this does not prohibit an analysis based on the tools which are usually used for multigrid schemes.

\section{Local Fourier analysis for PFASST}
\label{ch:lfa_and_transformation_matrices}
The most common tool for analysis and design of multigrid algorithms is the Local Fourier Analysis (LFA, see e.g.~\cite{trottenberg2000multigrid}).
It simplifies the problem by making assumptions like periodic domains and constant coefficients. 
The goal of LFA is, in the rigorous case, the computation and usually the estimation of the 
spectral radius of the iteration matrix and its building blocks.  

In this work we focus on two prototype problems, namely the diffusion and advection problem in one dimension,
to show how PFASST can be analyzed in principle. 
We will use periodicity in space to stay rigorous in that dimension.

The usual approach to LFA is to define and work with Fourier symbols for each operator. 
These Fourier symbols represent the behavior of the operators on the grid functions 
\begin{align}
  \varphi_{\theta}(x) = \exp\left( i \theta x / h\right),\quad x \in \left[ 0,1 \right],\ \theta \in \left[ -\pi, \pi \right),
  \label{eq:grid_functions}
\end{align}
for distinct frequencies $\theta$. 
The observation, how the different grid functions are damped or changed on different grids and under different operations is a central point of LFA.

However, in our analysis we will make use of the matrix notation and henceforth avoid the use of explicit Fourier symbols, 
but rather perform a block diagonalization of the matrices of PFASST.
The goal is the block-wise diagonalization of the iteration matrix of PFASST. 
Later on, each block will be associated with a discrete frequency. 
Therefore, we will be able to state which frequency is damped or changed to which extend. 

Due to the periodicity in space, parts of the iteration matrix consists of circulant matrices. 
A circulant matrix is a special kind of Toeplitz matrix where each row vector is rotated one element to the right relative to the preceding row vector
and denoted as
\begin{align}
   \matr{C} = \begin{pmatrix} c_{0} & c_1 & \cdots & c_{\frac{N}{2}-1} \\
    c_{\frac{N}{2}-1} & c_{0} &  & \\
                       & \ddots & \ddots &  \\
		       c_1 & \cdots & c_{\frac{N}{2}-1} & c_{0}
		     \end{pmatrix}.
  \label{eq:circulant_matrix}
\end{align}
It has the eigenvalues $\lambda_k$ and eigenvectors $\psi_k$  for $k = {0,\ldots,N-1}$
\begin{align}
  \begin{split}
    \lambda_k &= \sum_{j=0}^{N-1} c_j \exp\left(i\frac{2\pi}{N}k \cdot j\right) \quad \mbox{and} \\ 
   \psi_k &= \frac{1}{\sqrt{N}} \left[\exp\left(i\frac{2\pi}{N}k\cdot 0\right),\exp\left(i\frac{2\pi}{N}k\cdot 1\right), \ldots, \exp\left(i\frac{2\pi}{N}k\cdot (N-1)\right)\right]^T . 
  \end{split}
   \label{eq:eig_vals_circ_matr}
\end{align}
This also means that with the transformation matrix $\matr{\Psi}$, which is orthogonal and consists of the eigenvectors, it holds
\begin{align}
  \left(\matr{\Psi}^T \matr{C} \matr{\Psi} \right)_{j,j} = \lambda_j.
  \label{eq:spectral_transform}
\end{align}

For two diagonalizable matrices $\matr{A},\matr{B}$ with the same eigenvector space it holds:
\begin{align}
  \begin{split}
   \matr{\Psi}^T \left( \matr{A}+\matr{B}\right) \matr{\Psi} &=  \matr{\Psi}^T \matr{A} \matr{\Psi} + \matr{\Psi}^T \matr{B} \matr{\Psi} = \matr{D}^{(A)} + \matr{D}^{(B)},   \\
  \matr{\Psi}^T \matr{A}\matr{B}\matr{\Psi} &= \matr{\Psi}^T \matr{A} \matr{\Psi}\matr{\Psi}^T \matr{B}\matr{\Psi} = \matr{D}^{(A)} \matr{D}^{(B)},\\
  \matr{\Psi}^T \matr{A}^{-1}\matr{\Psi} &= \left(\matr{D}^{(A)}\right)^{-1} 
  \end{split}
  \label{eq:transformation_rules}
\end{align}
Furthermore, for the Kronecker product we have $\tmatr{P}^{-1} \matr{A} \otimes \matr{B} \tmatr{P} = \matr{B} \otimes \matr{A}$, where $\tmatr{P}$ is a suitable permutation matrix.
Those rules will be used extensively by the following algebraic manipulations.

\subsection{Transforming the PFASST iteration matrix}
\label{ch:transforming_pfasst}

The PFASST algorithm has 3 layers it works on. 
The first layer is the spatial space, the second consists of the quadrature nodes, and the third is the temporal structure given by the subintervals.
All layers are interweaved: we illustrate this by rewriting the system matrix $\matr{M}_{[ t_0,T]}$ under the assumption that we have the same problem (i.e.~the same discretization of the same operator) on each subinterval
\begin{align}
  \matr{M}_{[ t_0,T]} &= \matr{I}_L \otimes \matr{I}_M \otimes \matr{I}_N - \matr{I}_L \otimes \matr{Q} \otimes \matr{A} - \matr{E} \otimes \matr{N} \otimes \matr{I}_N,
  \label{eq:M_3_layers}
\end{align}
where $N$ is again the number of degrees of freedom in the spatial dimension, $M$ the number of nodes per subinterval, and $L$ the number of subintervals. 
Also, a new operator $\matr{E} \in \R^{L\times L}$ is introduced, which has ones on the first subdiagonal and zeros elsewhere.
In each term the layers are separated by the Kronecker product, and through the summation of those parts we interweave them again. 
Our transformation aims at the layer,
where each matrix is diagonalizable by $\matr{\Psi}$.
%\TODO{I do not understand this sentence - Die Transformation die wir konstruiren wollen konzentriert sich auf die Schicht wo jeder summand Operand in dieser Schicht mit der gleichen Transformation diagonalisiert werden kann, deswegen circulante matritzen}

We define a transformation matrix $\tmatr{F}$, which effects all layers, as
\begin{align*}
  \tmatr{F}= \tmatr{P} \cdot\left( \matr{I}_L \otimes \matr{I}_N \otimes \matr{\Psi}\right),\quad \tmatr{F}^{-1} = \left(\matr{I}_L \otimes \matr{I}_N \otimes \matr{\Psi}^{T} \right)\cdot \tmatr{P}^{-1}, 
\end{align*}
and therefore
\begin{align*}
  \tmatr{F}^{-1} \matr{M}_{[ t_0,T]} \tmatr{F}= \matr{I}_N \otimes \left( \matr{I}_{L}\otimes \matr{I}_{M} - \matr{E} \otimes \matr{N} \right)  - \matr{D}^{(A)}\otimes \matr{I}_L \otimes \matr{Q}.
\end{align*}

This yields diagonal matrices on the layer for the spatial dimension, so that we can write:

\begin{align*}
  \tmatr{F}^{-1} \matr{M}_{[ t_0,T]} \tmatr{F} &= \mathrm{diag}\left(\matr{B}^{(M_{[ t_0,T]})}_1, ..., \matr{B}_N^{(M_{[ t_0,T]})}\right)\\
  \quad \mbox{with} \quad \matr{B}^{(M_{[ t_0,T]})}_j &= \matr{I}_{L}\otimes \matr{I}_{M} - \matr{E} \otimes \matr{N} - \lambda_j \matr{I}_L \otimes \matr{Q}.
\end{align*}
We call the resulting blocks ``time collocation blocks'', highlighting the dimension and components of the blocks.
This idea was recently introduced in~\cite{friedhoff2015generalized} in a different notation and is named ``semi-algebraic mode analysis'' (SAMA). 
The motivation behind SAMA is the large gap between the theoretical analysis and the actual performance of multigrid methods for parabolic equations and tine-parallel methods.
In~\cite{friedhoff2015generalized} Friedhoff et al.~demonstrated that SAMA
%'s accounting for the boundary conditions and heterogeneity 
enables accurate predictions of the short-term behavior and asymptotic convergence factors. 

The transformation strategy above leads to a block structure for all matrices 
which emerge in the formulation of PFASST, in particular for the iteration matrix.
Here, the interpolation and restriction matrices need special attention, though.

\subsubsection{Transforming Interpolation and Restriction}

In this section we focus on interpolation and restriction operators, 
which are designed for two special isometric periodic grids with an even number of fine grid points.
Between these two grids we define a special class of interpolation and restriction pairs. 

\begin{definition}

  Let $\matr{C}\in \R^{N\times N}$ be a circulant matrix, 
  with the associated eigenvalues $\left\{ \lambda_k \right\}_{k=1\ldots \frac{N}{2}}$,
  and let the fine grid $X$ and coarse grid $ \tilde{X}$ be defined as  
\begin{align*}
  X = [x_1,\ldots, x_N] \; \mbox{ and } \; \tilde{X} = [\tilde{x}_1, \ldots, \tilde{x}_{N/2}] \mbox{ ,with }\; x_{2j-1} = \tilde{x}_{j} \; \mbox{for all }\; j\in \left\{ 1,\ldots, \frac{N}{2} \right\}.
\end{align*}
Let $\tmatr{W}(.,.):\R^{N\times N}\times\R^{N\times N} \mapsto \R^{2N \times N}$ be an ``interweaving'' operator, 
which stacks together the rows of two matrices subsequently, 
beginning with the first row of the first matrix, followed by the first row of the second matrix and finally ending with the last row of the second matrix.
Then we define the class of circulant interweaved interpolation (``CI-interpolation'') operators as
\begin{align}
  \Pi = \left\{ \matr{T}_C^F : \exists \matr{C} \in \R^{N\times N} \mbox{ circulant and  } \matr{C}\cdot \vect{1}=\vect{1},  \matr{T}_C^F = \tmatr{W}(\matr{I}_N, \matr{C})  \right\}
  \label{eq:harmless_interpolation}
\end{align}
and the class of circulant interweaved restriction (``CI-restriction'') operators as
\begin{align}
  \Pi^T = \left\{ \matr{T}_F^C : c \left( \matr{T}_F^C \right)^T \in \Pi,\ c\in\mathbb{R} \right\}.
  \label{eq:harmless_restriction}
\end{align}
\end{definition}
Due to the circulant nature of the interweaved matrices, we are able to state a transformation analytically.
\begin{lemma}
  Let $\matr{T}_C^F$ be a  CI-interpolation and $\matr{T}_F^C$ the associated CI-restriction operator, $\matr{\Psi}$ the transformation matrix for $N$ grid points and
  $\matr{\Psi}_C$ the transformation matrix for $N/2$ grid points. 
Then it holds
  \begin{align}
    \matr{\Psi}^T \matr{T}_C^F \matr{\Psi}_C &=
  \begin{pmatrix}
	d_0  & & \\
	     & \ddots & \\
	     &       & d_{N/2-1} \\
 \hat{d}_0   & &   \\
	    &  \ddots & \\
	    &  & \hat{d}_{N/2-1} 
	\end{pmatrix}
	\label{eq:transformed_intpl}
	\end{align}
and
    \begin{align}
\matr{\Psi}^T_C \matr{T}_F^C \matr{\Psi}
&= \frac{1}{2}
\begin{pmatrix}
  d_0  & &  &\hat{d}_0   & & \\
    &  \ddots & &    &  \ddots &  \\
    &        & d_{N/2-1}     & & & \hat{d}_{N/2-1}  \\
\end{pmatrix}.
    \label{eq:transformed_rstr}
  \end{align}
  The values on the diagonal depend solely on the circulant matrix $\matr{C}$ and its eigenvalues $\lambda^{(C)}_k$ for $k \in \left\{ 0,N/2-1 \right\}$. More precisely, we have
\begin{align}
  d_{k} = \frac{1+\lambda^{(C)}_k \exp(-i\frac{2\pi}{N}k)}{\sqrt{2} } \quad \mbox{and} \quad \hat{d}_{k} = \frac{1-\lambda^{(C)}_k \exp(-i\frac{2 \pi}{N}k)}{\sqrt{2} }.
  \label{eq:diagonal_entries_t_intpl}
\end{align}
\label{lem:transformation_intpl_rstr}
\end{lemma}
\begin{proof}
 Using the properties of the interweaving operator we have
 \begin{align*}
 \matr{T}_C^F\cdot\matr{\Psi}_C = \tmatr{W}(\matr{I}_{\frac{N}{2}}, \matr{C}_{\frac{N}{2}})\matr{\Psi}_C = \tmatr{W}(\matr{\Psi}_C , \matr{C}_{\frac{N}{2}}\matr{\Psi}_C ).
 \end{align*}
 Using the eigenvector eigenvalue relation~\eqref{eq:eig_vals_circ_matr} of the two circulant matrices $\matr{C}$ and $\matr{I}$, see Section~\ref{ch:lfa_and_transformation_matrices}, 
 for the computation of
 \begin{align*}
   \left[ \matr{T}_C^F\cdot\matr{\Psi}_C \right]_{-,k}  &= 
   \sqrt{\frac{2}{N}}
 \begin{pmatrix}\exp(i 4 \pi/N k\cdot 0)  \\  \lambda_k \exp(i 4 \pi/N k\cdot 0) \\ \vdots \\ \exp(i 4 \pi/N k\cdot (N/2-1))  \\  \lambda_k \exp(i 4 \pi/N k\cdot (N/2-1))\end{pmatrix}.
 \end{align*}
 A comparison to the immediate meaning of~\eqref{eq:transformed_intpl} demands

  \begin{align*}
   \left[ \matr{T}_C^F\cdot\matr{\Psi}_C \right]_{-,k}  
 & \mathop{=}\limits^! \frac{d_k}{\sqrt{N}} \begin{pmatrix} \exp(i 2\pi/Nk\cdot0)\\ \vdots \\ \exp(i 2\pi/Nk\cdot(N-1)) \end{pmatrix} +\frac{\hat{d}_{k}}{\sqrt{N}} \begin{pmatrix} \exp(i 2\pi/N(N/2+k)\cdot0) \\ \vdots \\ \exp(i 2\pi/N (N/2+k) \cdot(N-1)) \end{pmatrix}.
 \end{align*}
Solving this system yields \eqref{eq:diagonal_entries_t_intpl}.
\end{proof}

Depending on the structure of $\matr{C}$, we are able to state further simplifications for
$d_k$ and $\hat d_k$, as we see in the following remark.

\begin{remark}
 For the special cases where $\matr{C}$ has a symmetric stencil with 
 \begin{align*}
   c_l = 
   \begin{cases}
     c_{\frac{N}{2} - l},  &\mbox{stencil length odd and } l \in {1,\ldots,m} \\
     c_{\frac{N}{2} - (l+1)}, &\mbox{stencil length even and } l \in {0,\ldots,m-1} \\
     0, &l > m \\
   \end{cases},
 \end{align*}
where $m \leq N/4$ for the even and $m \leq N/4 - 1/2$ for the odd case. 
Then, it holds for the odd case
$ d_k = d_{N/2-k} $ and $ \hat d_k = \hat d_{N/2-k} $.
In addition, for a CI-interpolation and -restriction operator with $\matr{C}\cdot \vect{1}=\vect{1}$ we have that $\lambda^{(C)}_0 = 1$ and hence $d_0 = 0$ and $\hat d_0 = \sqrt{2}$.
\label{rem:symmetric_stencils}
\end{remark}

We now use Lemma \ref{lem:transformation_intpl_rstr} to transform the coarse-grid correction.
For the interpolation operator, we obtain diagonal entries $\left\{ d_0,\hat d_0, \ldots, d_{N/2-1}, \hat d_{N/2-1}  \right\}$ and for the restriction operator the diagonal entries $\left\{ f_0,\hat f_0, \ldots, f_{N/2-1}, \hat f_{N/2-1}  \right\}$.
These entries may coincide if the same circulant matrix $\matr{C}$ is used for the construction of both operators.
Furthermore, we transform the inverse of the system matrix $\matr{\tilde{A}}^{-1}$
in the spatial dimension into a diagonal matrix consisting of the eigenvalues $\left\{ \tilde \lambda_0, \ldots, \tilde \lambda_{N/2 - 1} \right\}$ of $\matr{\tilde{A}}^{-1}$. 
Then we obtain
\begin{multline*} %\label{eq:permutate_cg}
   \matr{\Psi}^T \matr{T}_C^F \matr{\tilde{A}}^{-1} \matr{T}_F^C \matr{\Psi} 
   = \matr{\Psi}^T \matr{T}_C^F \matr{\Psi}_C\matr{\Psi}_C^T  \matr{\tilde{A}}^{-1} \matr{\Psi}_C\matr{\Psi}_C^T\matr{T}_F^C \matr{\Psi} \\
   \begin{aligned}
  &=  \frac{1}{2}
  \begin{pmatrix}
	d_0  & & \\
	     & \ddots & \\
	     &       & d_{\frac{N}{2}-1} \\
 \hat{d}_0   & &   \\
	    &  \ddots & \\
	    &  & \hat{d}_{\frac{N}{2}-1} 
	\end{pmatrix}
	\begin{pmatrix}
	  \tilde{\lambda}_0 & & \\
	  & \ddots & \\
	  & &  \tilde{\lambda}_{\frac{N}{2}-1} 
	\end{pmatrix}
\begin{pmatrix}
  f_0  & &  &\hat{f}_0   & & \\
    &  \ddots & &    &  \ddots &  \\
    &        & f_{\frac{N}{2}-1}     & & & \hat{f}_{\frac{N}{2}-1}  \\
\end{pmatrix} \\
 &= 
 \frac{1}{2}
 \begin{pmatrix}
   d_1 \tilde{\lambda}_0 f_0 & & & \hat{d}_0 \tilde{\lambda}_0 f_0 & & \\
	& \ddots & &                 & \ddots & \\
	& & d_{\frac{N}{2}-1} \tilde{\lambda}_{\frac{N}{2}-1} f_{\frac{N}{2}-1} & & &      & \tilde{d}_{\frac{N}{2}-1} \tilde{\lambda}_{\frac{N}{2}-1} f_{\frac{N}{2}-1} \\
	d_0 \tilde{\lambda}_0 \hat{f}_0 & & & 	\hat{d}_0 \tilde{\lambda}_0 \hat{f}_0  & & \\
	& \ddots & &                 & \ddots & \\
	& & d_{\frac{N}{2}-1} \tilde{\lambda}_{\frac{N}{2}-1}\hat{f}_{\frac{N}{2}-1}  & & &      & \hat{d}_{\frac{N}{2}-1} \tilde{\lambda}_{\frac{N}{2}-1} \hat{f}_{\frac{N}{2}-1} 
 \end{pmatrix}.
 \end{aligned}
\end{multline*}
The values are now scattered over $3$ diagonals.
By using the appropriate permutation matrix we can gather them to new blocks:
\begin{align}
  \tmatr{P}^{-1}\matr{\Psi}^T \matr{T}_C^F \matr{\tilde{A}}^{-1} \matr{T}_F^C \matr{\Psi} \tmatr{P} &=  \mbox{diag}\left( \matr{B}_0, \ldots, \matr{B}_{\frac{N}{2}-1} \right),\\
  \mbox{where}\quad  \matr{B}_l &= 
  \begin{pmatrix}
    d_l \tilde{\lambda}_l f_l & \hat{d}_l \tilde{\lambda}_l f_l  \\
    d_l \tilde{\lambda}_l \hat{f}_l & \hat{d}_l \tilde{\lambda}_l \hat{f}_l \\
  \end{pmatrix} \in \R^{2\times 2}.
 \label{eq:permutate_cg_2}
\end{align}
In this structure we find the classical mode-mixing property of interpolation and restriction operators. 
This well-known property of standard multigrid iterations interweaves pairs of one low and one high frequency, the``harmonics''. 

\subsubsection{Transforming the full iteration matrix}

The iteration matrix of PFASST can now be transformed into a block matrix with $N/2$ blocks of the size $M\cdot L$. 
Each block is associated with a harmonic of the spatial problem and therefore with one high and one low frequency. 
In contrast, the smoother alone is decomposed into $N$ blocks, 
which may be associated with only one single frequency. 
This is summarised in the following theorem. 

\begin{theorem}
  Let us have a iteration matrix in the form of \eqref{eq:iteration_matrix_pfasst} with
  \begin{align*}
    \matr{T} &= \left( \matr{I} - \matr{P}^{-1}\matr{M} \right)\left( \matr{I} - \matr{T}_C^F \matr{\tilde{P}}^{-1} \matr{T}_C^F \matr{M}  \right),
  \end{align*}
  where $\matr{M}$ is the collocation matrix, $\matr{T}_F^C,\matr{T}_C^F$ are two circulant interweaved transfer operators and $\matr{P},\matr{\tilde{P}}$ are two preconditioner with 
  a matrix in the spatial layer, which is diagonalisable and has the same eigenvector space as the spatial system matrix $\matr{A}$. 
  Then there exists a transformation $\tmatr{F}$ so that
  \begin{align}
    \tmatr{F}^{-1} \matr{T} \tmatr{F} &=  \mathrm{diag}\left( \tmatr{B}^{(S)}_0 \tmatr{B}^{(CGC)}_0, \ldots, \tmatr{B}^{(S)}_{\frac{N}{2}-1}\tmatr{B}^{(CGC)}_{\frac{N}{2}-1}
  \right) \in \R^{LMN\times LMN}\quad ,\mbox{with} \label{eq:transformed_iteration_matrix_1}
\\
    \tmatr{B}^{(S)}_k &= 
     \begin{pmatrix}
       \matr{I} - \left(\matr{B}_k^{(P)}\right)^{-1} \matr{B}^{(M)}_k & \\
       & \matr{I} - \left(\matr{B}_{\frac{N}{2} + k}^{(P)}\right)^{-1} \matr{B}^{(M)}_{\frac{N}{2} + k} 
     \end{pmatrix} \in \R^{2LM \times 2LM} \label{eq:transformed_iteration_matrix_2}
 \\
\tmatr{B}^{(CGC)}_k &=
     \begin{pmatrix} 
       \matr{I} - f_k d_k \left(\matr{B}^{(\tilde{P})}_k\right)^{-1} \matr{B}^{(M)}_k & -\hat{f}_k d_k \left(\matr{B}^{(\tilde{P})}_k \right)^{-1}\matr{B}^{(M)}_{\frac{N}{2}+k} \\
       -\hat{d}_k f_k \left(\matr{B}^{(\tilde{P})}_k \right)^{-1} \matr{B}^{(M)}_{k} & \matr{I} - \hat{f}_k \hat{d}_k \left(\matr{B}^{(\tilde{P})}_k \right)^{-1} \matr{B}^{(M)}_{\frac{N}{2}+k}
     \end{pmatrix} \in \R^{2LM \times 2LM} ,
    \label{eq:transformed_iteration_matrix_3}
  \end{align}
  with matrices $B_k^{(P)},B_k^{(M)} \in \R^{LM\times LM}$ for $k=0\ldots N-1$ and  $B_k^{(\tilde P)} \in \R^{LM\times LM}$  for $k=0\ldots \frac{N}{2} - 1$, 
  solely depending on the eigenvalues of $\matr{A}$ and $\matr{\tilde A}$. Where
  \begin{align}
    \begin{split}
      \matr{\Psi}^T \matr{P} \matr{\Psi} &= \mathrm{diag}\left( \matr{B}^{(P)}_0, \ldots , \matr{B}^{(P)}_{N-1} \right), \mbox{with }\;
      \matr{B}^{(P)}_j = \matr{I}_{L}\otimes \matr{I}_{M} - \matr{E} \otimes \matr{N} - \lambda_j^{(A)} \matr{I}_L \otimes \matr{Q_{\Delta}},\\
    \matr{\Psi}^T \matr{M} \matr{\Psi} &= \mathrm{diag}\left( \matr{B}^{(M)}_0, \ldots , \matr{B}^{(M)}_{N-1} \right), \mbox{with }\;
    \matr{B}^{(M)}_j = \matr{I}_{L}\otimes \matr{I}_{M} - \matr{E} \otimes \matr{N} - \lambda_j^{(A)} \matr{I}_L \otimes \matr{Q},\\
    \matr{\Psi}^T \matr{\tilde{P}} \matr{\Psi} &= \mathrm{diag}\left( \matr{B}^{(\tilde{P})}_0, \ldots , \matr{B}^{(\tilde{P})}_{\frac{N}{2}-1} \right), \mbox{with }\;
    \matr{B}^{(\tilde{P})}_j= \matr{I}_{L}\otimes \matr{I}_{M} - \matr{E} \otimes \matr{N} - \lambda_j^{(\tilde{A})} \matr{I}_L \otimes \matr{Q_{\Delta}}.
    \label{eq:basic_blocks_sama}
    \end{split}
  \end{align}
%Let $\left[ \cdot \right]_{\mathbf{j}}$ denote the $j$-th block of the block-matrix, then it holds
%  \begin{align}
%    \begin{split}
%    \matr{B}^{(P)}_j = \left[\matr{\Psi}^T \matr{P} \matr{\Psi}\right]_{\mathbf{j}} &= \matr{I}_{L}\otimes \matr{I}_{M} - \matr{E} \otimes \matr{N} - \lambda_j^{(A)} \matr{I}_L \otimes \matr{Q_{\Delta}}\\
%    \matr{B}^{(M)}_j = \left[\matr{\Psi}^T \matr{M} \matr{\Psi}\right]_{\mathbf{j}}&= \matr{I}_{L}\otimes \matr{I}_{M} - \matr{E} \otimes \matr{N} - \lambda_j^{(A)} \matr{I}_L \otimes \matr{Q}\\
%    \matr{B}^{(\tilde{P})}_j = \left[\matr{\Psi}^T \matr{\tilde{P}} \matr{\Psi} \right]_{\mathbf{j}}&= \matr{I}_{L}\otimes \matr{I}_{M} - \matr{E} \otimes \matr{N} - \lambda_j^{(\tilde{A})} \matr{I}_L \otimes \matr{Q_{\Delta}}.
%    \label{eq:basic_blocks_sama}
%    \end{split}
%  \end{align}
  We call $\matr{B}^{(M)}_j, \matr{B}^{(P)}_j$ and $\matr{B}^{(\tilde{P})}_j$ basic blocks.
  The matrix $\matr{Q}_{\Delta} \in \R^{M \times M}$ is a lower triangular matrix approximating $\matr{Q}$, see Section \ref{ch:sdc}.
\label{th:transformed_iteration_matrix}
\end{theorem}
\begin{proof}
The proof consists of straightforward computations.
The matrices $\matr{P},\matr{M}$ and $\matr{\tilde{P}}$ have 3 layers, 
separated by Kronecker products like in \eqref{eq:M_3_layers}. 
Applying the transformation in the spatial dimension leads to the basic blocks \eqref{eq:basic_blocks_sama}.
Similar to \eqref{eq:permutate_cg_2}, we choose the adequate permutation matrices on the 
layers of subintervals and quadrature nodes, to get the blocks of harmonics.
Also, each block of the post smoother is associated with a mode, hence we stack harmonic pairs together to $\tmatr{B}^{(S)}_k$ in order to match them with the blocks of the coarse grid correction $\tmatr{B}^{(CGC)}_k$,
performed by the same permutation matrix.
% The smoother is easily transformed into a block matrix with blocks
% \begin{align*}
%   \matr{B}^{(S)}_{j} = \matr{I} - \left(\matr{B}_j^{(P)}\right)^{-1} \matr{B}^{(M)}_j \quad \mbox{,for}\; j \in \left\{ 0, \ldots, N-1 \right\}.
% \end{align*}
\end{proof}

This theorem makes it possible, at least semi algebraically, to analyze
the convergence properties of PFASST by computing the spectral radius of each block $\tmatr{B}^{(S)}_k\tmatr{B}^{(CGC)}_k$.
Until this point the choice of the particular problem and the operators yields a rigorous transformation. 
Hence, the blocks and the full iteration matrix of PFASST have exactly the same eigenvalues.
This translates to computing $N/2$ eigenvalues of matrices of the size $2ML \times 2ML$.
As we can see in \eqref{eq:basic_blocks_sama}, the basic blocks consists of $\matr{I}_L$ and $\matr{E}$ on the first layer. 
However, it is not directly possible to apply the transformation strategy presented above to this layer.
For an empirical study like LFA, though, only estimates of the spectral radii are needed. 
This is mainly due to the fact, that even the exact spectral radius does not reflect the direct numerical behavior of the method exactly, but rather asymptotically.
In the following section we therefore give up the rigorousness of the transformation in order to find a decomposition of the basic blocks into $L$ blocks of the size $2M \times 2M$.
%which also will be used for the observation of our method.

\subsection{Assuming periodicity in time}
\label{ch:periodicity_in_time}

To enable the further decomposition of the basic blocks, we exchange in the matrix formulation 
\begin{align*}
  \matr{E}=
    \begin{pmatrix}
      0 & 0 & \cdots & 0 \\
      1 & 0 \\
      \vdots & & \ddots\\
      0 & & 1 & 0 
    \end{pmatrix} \quad \mbox{with} \quad 
    \matr{\hat{E}}=
    \begin{pmatrix}
      0 & 0 & \cdots & 1 \\
      1 & 0 \\
      \vdots & & \ddots\\
      0 & & 1 & 0 
    \end{pmatrix},
\end{align*}
which introduces time periodicity to the problem and makes the matrix circulant. 
Hence, it becomes easy to transform
\begin{align*}
  \left[  \matr{\Psi}^{-1} \matr{\hat{E}} \matr{\Psi} \right]_{j,j} = \exp\left(-i2\pi \frac{j}{L}\right),
\end{align*}
which makes the basic blocks $\matr{B}^{(M)}_j,\matr{B}^{(P)}_j$ and $\matr{B}^{(\tilde{P})}$  further decomposable into $NL$ or $NL/2$ blocks of the size $M \times M$ or $2M\times 2M$, respectively.
This leads directly to the following Theorem that can be proved using straightforward computations similar to the ones used before.
\begin{theorem}
  Let us have the identical requirements as in Theorem~\ref{th:transformed_iteration_matrix}, except the use of $\matr{\hat{E}}$ instead of $\matr{E}$.
  Then there exists a transformation $\tmatr{F}$ such that
\begin{align}
  \tmatr{\hat F}^{-1} \matr{T} \tmatr{\hat F} &= 
  \mathrm{diag}\left( 	\tmatr{B}_{0,0}^{(S)} \cdot \tmatr{B}_{0,0}^{(CGC)},
  			\tmatr{B}_{0,1}^{(S)} \cdot \tmatr{B}_{0,1}^{(CGC)}, \ldots,
			\tmatr{B}_{\frac{N}{2}-1,L-1}^{(S)} \cdot \tmatr{B}_{\frac{N}{2}-1,L-1}^{(CGC)} \right) ,\\ \mbox{with}\quad
  \tmatr{B}_{k,j}^{(S)} &= 
  \begin{pmatrix}
    \matr{I} - \left(\matr{B}_{k,j}^{(P)}\right)^{-1} \matr{B}^{(M)}_{k,j} & \\
       & \matr{I} - \left(\matr{B}_{\frac{N}{2} + k,j}^{(P)}\right)^{-1} \matr{B}^{(M)}_{\frac{N}{2} + k,j} 
     \end{pmatrix} \in \R^{2M \times 2M} \;\mbox{and}\\
  \tmatr{B}_{k,j}^{(CGC)} &= 
  \begin{pmatrix}
    \matr{I} - f_k d_k \left( \matr{B}^{(\tilde{P})}_{k,j} \right)^{-1} \matr{B}^{(M)}_{k,j} & -\hat{f}_k d_k\left( \matr{B}^{(\tilde{P})}_{k,j}\right)^{-1} \matr{B}^{(M)}_{N/2+k,j} \\
    -\hat{d}_k f_k \left(\matr{B}^{(\tilde{P})}_{k,j} \right)^{-1} \matr{B}^{(M)}_{k,j} & \matr{I} - \hat{f}_k \hat{d}_k\left( \matr{B}^{(\tilde{P})}_{k,j}\right)^{-1} \matr{B}^{(M)}_{N/2+k,j}
  \end{pmatrix} \in \R^{2M \times 2M},
  \label{eq:block_lfa_iteration_matrix}
\end{align}
with matrices $B_{k,j}^{(P)},B_{k,j}^{(M)} \in \R^{M\times M}$ for $k=0\ldots N-1, j= 0\ldots L-1$ and  $B_{k,j}^{(\tilde P)} \in \R^{M\times M}$  for $k=0\ldots \frac{N}{2} - 1, j=0\ldots L-1$, 
solely depending on the eigenvalues of $\matr{A}$ and $\matr{\tilde A}$, with
  \begin{align}
    \begin{split}
      \matr{\Psi}^T \matr{P} \matr{\Psi} &= \mathrm{diag}\left( \matr{B}^{(P)}_{0,0}, \ldots , \matr{B}^{(P)}_{N-1,L-1} \right), \mbox{with }\;
    \matr{B}^{(P)}_{k,j} = \matr{I} - \lambda_k^{(A)} \Delta t \matr{Q}_{\Delta} +  \exp\left(-i2\pi \frac{j}{L}\right) \matr{N},\\
    \matr{\Psi}^T \matr{M} \matr{\Psi} &= \mathrm{diag}\left( \matr{B}^{(M)}_{0,0}, \ldots , \matr{B}^{(M)}_{N-1,L-1} \right), \mbox{with }\;
    \matr{B}^{(M)}_{k,j} = \matr{I} - \lambda_k^{(A)} \Delta t \matr{Q} +  \exp\left(-i2\pi \frac{j}{L}\right)\matr{N},\\
    \matr{\Psi}^T \matr{\tilde{P}} \matr{\Psi} &= \mathrm{diag}\left( \matr{B}^{(\tilde{P})}_{0,0}, \ldots , \matr{B}^{(\tilde{P})}_{\frac{N}{2}-1,L-1} \right), \mbox{with }\;
    \matr{B}^{(\tilde{P})}_{k,j}=  \matr{I} - \lambda_k^{(\tilde{A})} \Delta t \matr{Q}_{\Delta} +  \exp\left(-i2\pi \frac{j}{L}\right) \matr{N}.
    \label{eq:block_lfa_blocks}
    \end{split}
  \end{align}
We denote those blocks as ``collocation blocks'' in contrast to the time-collocation blocks of Theorem~\ref{th:transformed_iteration_matrix}.
  \label{th:transformation_iteration_matrix_blfa}
\end{theorem}
%\begin{proof}
% Straightforward computations. 
%\end{proof}

This leaves us with $NL/2$ blocks of the size $2M\times 2M$.
We identify the matrices $\matr{Q}$ and $\matr{Q}_{\Delta}$ as the atomic part of the whole matrix formulation.
Further decompositions may only be performed if a decomposition of $\matr{Q}$ is found.  
In the case of a $\matr{Q} \in \R^{1 \times 1}$, 
the time stepping part reduces to e.g. an implicit Euler.
In this case no eigenvalue computations are necessary any more and 
the Fourier symbols are easily derived from the basic collocation blocks.  
%\TODO{ref Neumueller? - Wurde schon vor Neumueller gemacht, ausserdem muesste man das DG Fass aufmachen}. 

\begin{remark}
With the assumption of periodicity in time we loose the initial value,
which means that if $u(t,x)$ is a solution of the problem then $u(t,x)+c$ is also a solution for any $c \in \R$. 
Hence, the inverses of $\matr{B}_{k,0}^{(P)}$ and $\matr{B}_{k,0}^{(\tilde{P})}$ do not exist and neither do the inverses of iteration matrix blocks $\tmatr{B}_{k,0}^{(T)}=\tmatr{B}_{k,0}^{(S)}\cdot\tmatr{B}_{k,0}^{(CGC)}$ exist. 
Our remedy for this problem is to set $\tmatr{B}_{k,0}^{(T)}$ to $\matr{0}$. 
This blocks belong to constant modes and we assume that there are no constant error modes which have to be damped.
  \label{rm:periodicity_in_time}
\end{remark}

Based on this transformation of the iteration matrix, 
we are now able to investigate the behavior of PFASST for two standard
model problems in the following section.

\section{Numerical Experiments}
\label{ch:numerical_experiments}
In this section we show how the convergence properties of PFASST may be examined along the lines of two examples, namely the diffusion and the advection problem.
Within this paper though, a full analysis of the influence all the parameters like $N$, $L$, $M$,  $\Delta t$, the choice of the quadrature rule or the PDE parameters is not possible.
Therefore, the experiments presented here do not aim for a complete analysis, they should rather be viewed as a recipe to analyze PFASST for a certain class of problems, defined by the requirements we posed for the theoretical results above.

All computations are performed with $N=128$ degrees of freedom in space.
For matrices and vectors, the infinity norm is used. 
For the advection problem, 
we will use the SDC algorithm with the LU-based preconditioner $\matr{Q}_\Delta$ as in~\cite{Weiser2013}, while for the diffusion problem $\matr{Q}_\Delta$ is the standard implicit Euler method.
For all experiments we use $M = 5$ Gau\ss-Radau nodes on each of the $L=4$ subintervals of the length $dt=0.1$. 
Hence, we have $T=0.4$. 
The interpolation is constructed such that polynomials up to order $6$ are interpolated exactly and the restriction is gained from an interpolation operator which interpolates polynomials up to the order of $2$.
%\TODO{Define ALL parameters to repeat the tests: $M=?$ Gauss-Radau? $T=?$, $L=?$... may need to define these for the different examples, but this is very important to do! - Done plus added Interpolation and restriction}

Our main goal will be the estimation of the error by using the block form of the iteration matrix of PFASST. 
With blocks $\matr{B}_k$, the computation of the norm of a matrix $\matr{H}$ reduces to
\begin{align}
  \norm{\matr{H}}^{2} = \sup_{\vect{x}_k \neq 0} \frac{\sum_{k=1}^{m} \norm{\matr{B}_k\vect{x}_k}^{2}}{\sum_{k=1}^{m} \norm{\vect{x}_k}^{2}}
  = \max_k\sup_{\vect{x}_k \neq 0} \frac{\norm{\matr{B}_k\vect{x}_k}^{2}}{\norm{\vect{x}_k}^{2}}
  = \max_k\norm{\matr{B}_k}^{2},
  \label{eq:norm_and_blocks}
\end{align}
see \cite{trottenberg2000multigrid} for a proof.
The same holds for the computation of the spectral radii.
In addition, the effort of computing the eigenvalues 
of $N/2$ time-collocation blocks of the size $2LM\times 2LM$ is obviously less than
for a $MLN\times MLN$ matrix. 
With the assumption in Section \ref{ch:periodicity_in_time} 
it even reduces to the computation of $NL/2$ collocation blocks of the size $2M \times 2M$.

 For both cases (time collocation and collocation blocks) we consider the following strategies for the estimation of the error vector $e^{\kappa}$ of the $\kappa$ iteration 
\begin{enumerate}
  \item use the spectral radius $\rho(\matr{T})$ of the iteration matrix
  \item use the norm of the iteration matrix $\norm{\matr{T}}$ 
  \item use the norm of the $\kappa$-th potency of the iteration matrix $\norm{\matr{T}^\kappa}$ 
  \item apply $\kappa$-th times the iteration matrix to the known error vector 
\end{enumerate}

The first strategy is based on the inequality for consistent matrix norms $\|\cdot\|$ and
each $\kappa \in \N$ 
\begin{align}
  \rho(A)\leq \|A^{\kappa}\|^{\frac{1}{\kappa}},
  \label{eq:spec_rad_estimate}
\end{align}
see \cite{kelleyc1995iterative}.
Strategies 2 and 3 rely on the inequality
\begin{align}
  \norm{\vect{e}^{\kappa}} = \norm{\matr{T}^{\kappa}\vect{e}^0} \leq \norm{\matr{T}^{\kappa}}\norm{\vect{e}^0} \leq \norm{\matr{T}}^{\kappa}\norm{\vect{e}^0}.
  \label{eq:norm_estimates}
\end{align}
Note that the iteration matrix is separated from the initial error vector $\vect{e}^0$, 
and therefore an a priori estimation of the relative error reduction is possible for this strategies.
In contrast, the error vector $\vect{e}^0$, i.e.~the analytical solution has to be known for strategy 4, making it an a posteriori strategy.
If time collocation blocks are used, the computation following strategy $4$ yields the analytically correct error for each iteration.
Using collocation blocks, this approach just provides another estimate.

\subsection{Diffusion problem}
\label{ch:diffusion_problem}

The elliptic Poisson problem is often used in the multigrid literature 
to  demonstrate the basic ideas of multigrid, see e.g.~\cite{trottenberg2000multigrid}.
Hence, the time-dependent, parabolic version of it, i.e.~the classical heat equation, is a canonical candidate for the analysis of a multigrid-like time integration method like PFASST.

The problem in one spatial dimension is given by
\begin{align}
  \begin{split}
  u_t &= \nu \Delta u, \quad x \in [0,1] \mbox{ and } t \in [0,T]\\
  u(x,0) &= u_0(x), \quad u(0,t) = u(1,t), \quad t \in [0,T]
  \end{split}
  \label{eq:heat_eq}
\end{align}
for a time $T>0$ and the diffusion coefficient $\nu>0$.
Using second-order finite differences on a isometric grid we get a simple discretization in the spatial dimension with
\begin{align}
  X = [x_1, \ldots , x_N], \mbox{ with } x_j = \frac{j-1}{N} \mbox{ and } \Delta x = \frac{1}{N}
\end{align}
which leads to a system of linear ODEs
\begin{align}
    \begin{split}
  \xvect{U}_t(t)&= \matr{A}\xvect{U}(t), \quad t \in [0,T] \mbox{ and } \xvect{U}(0) = [u(x_1,0),\ldots, u(x_N,0)],  \\
  \mbox{with} \quad \matr{A} &= \frac{\nu}{\left( \Delta x \right)^2}
  \begin{pmatrix}
    2 & -1 & 0 &\cdots & -1 \\
    -1& 2  & -1 & & \\
     0 & \ddots & \ddots & \ddots & 0\\
     \vdots &  & -1 & 2& -1\\
      -1 & 0 &  \cdots& -1 & 2 
    \end{pmatrix} \in \R^{N \times N}.
  \end{split}
  \label{eq:discrete_heat_eq}
\end{align}
Because the matrix $\matr{A}$ is circulant, 
the spectral decomposition in eigenvalues and eigenvectors is easily computed. 
For the eigenvalues $\lambda_k$ and 
normal eigenvectors $\psi_k$, $k\in \left\{ 0,\ldots, N-1 \right\}$, we have
\begin{align}
  \lambda_k = \frac{4\nu}{\Delta x^2} \sin^2\left(\frac{k\pi}{N}\right) \quad \mbox{and} \quad 
     \psi_k = \frac{1}{\sqrt{N}} \left[\exp\left(i\frac{2\pi}{N}k\cdot 0\right),\ldots, \exp\left(i\frac{2\pi}{N}k\cdot (N-1)\right)\right]^T.
  \label{eq:eig_vals_heat_eq}
\end{align}

\subsubsection{The error vector}

For initial values given by the function 
\begin{align}
  u_0(x) = \sin\left( 2 \pi x k \right), \quad k \in \left\{ 1,\ldots,N-1 \right\},\quad x \in \left[ 0,1 \right],
  \label{eq:heat_eq_simple_initial_functions}
\end{align}
we know that solution to our PDE with periodic boundary conditions is given by
\begin{align*}
  u(t,x) = \exp\left( -\nu \left( 2\pi k \right)^{2}t \right)\sin\left( 2\pi k x \right).
\end{align*}
Usually the PFASST algorithms starts with an vector where the initial value is spread on each node, i.e. we have the initial error vector
\begin{align*}
  \vect{e}^{0} = \begin{pmatrix}1 - \exp(-\nu(2\pi k)^2t_0 + \tau_1) \\ \vdots \\1-\exp(-\nu(2\pi k)^2 T) \end{pmatrix}\otimes\begin{pmatrix} \sin(2\pi kx_1) \\ \vdots \\ \sin(2\pi k x_{N}) \end{pmatrix}.
\end{align*}
With the iteration matrix we compute the succeeding error vector for PFASST as
\begin{align*}
  \matr{T} \vect{e}^{\kappa} = \vect{e}^{\kappa + 1}.
\end{align*}
Like the iteration matrix, the error vector $\vect{e}^k$ of the $k$-th iteration itself can be transformed and decomposed into parts belonging
to a certain mode and associated with the TC-block of the iteration matrix. 
We can write
\begin{align*}
  \tmatr{F}^{-1} \matr{T} \tmatr{F} \tmatr{F}^{-1} \vect{e}^{\kappa} = \tmatr{F}^{-1}\vect{e}^{\kappa + 1}.
\end{align*}
The transformed error is thus $ \vect{\hat{e}}^{\kappa} =  \tmatr{F}^{-1} \vect{e}^{\kappa} $, 
following precisely the transformation procedure 
described in Section~\ref{ch:transforming_pfasst}. 
The initial value function $u_0(x)$ decomposes into two modes, 
which are represented by spatial Fourier space functions
\begin{align*}
  \sin(2\pi kx_j) &= \frac{1}{2i} \exp\left(i \frac{2\pi}{N} kx_j\right) - \frac{1}{2i} \exp\left(-i \frac{2\pi}{N} kx_j\right) \\
  &= \frac{\sqrt{N}}{2i}\left[ \psi_k  \right]_j - \frac{\sqrt{N}}{2i}\left[  \psi_{N-k} \right]_j 
\end{align*}
and belong to two different harmonics. This reduces our analysis to the blocks belonging to these certain harmonics, which are $\tmatr{B}^{(T)}_{k},\tmatr{B}^{(T)}_{\frac{N}{2}-k}$ of sizes $2LM$ for $k=0, ..., N/2-1$ in the case time-collocation blocks are considered and $\tmatr{B}^{(T)}_{k,j}, \tmatr{B}^{(IT)}_{\frac{N}{2}-k,j}$ of sizes $2M$ for $k=0, ..., N/2-1$ and $j=0, ..., L-1$ if collocation blocks are considered. 

\subsubsection{Error prediction}
\label{ch:error_prediction}
%Previous computation showed that the norms and spectral radii of collocation blocks and time collocation blocks are more similar to each other for most operators, 
%if $\mu = \nu \frac{\Delta t}{\Delta x^2} > 1$.
%Hence, we choose $\nu$ so that $\mu = 10$. This improves the estimations, when collocation blocks are used \TODO{a dangerous game to play.. I would not state it like this. Just say that we choose $\mu=10$ and that's it. You already argued at the beginning that this is just an example..}

We choose $\nu$ so that $\mu = 10$.
In Figure \ref{fig:res_finding_2}, observing the solid line of the actual error measured during the iterations, we first see a short-term convergence behavior until roughly $10^{-4}$ which is then followed by a much slower, long-term convergence phase.
We see that the use of the norm of the $k$-th potency of the iteration matrix $\matr{T}$ (strategy 3) is well-suited to capture the long-term convergence behavior.
This is of course also true for strategy 1, using the spectral radius of the iteration matrix. 
Similar plots for various initial value functions $\sin\left( 2\pi kx \right)$ for $k\in \left\{ 1,\ldots, N/2-1 \right\}$ , 
were inspected and showed the same behavior for the long-term convergence. 
In particular, there is no significant difference between time-collocation and collocation blocks.
However, the norm of the iteration matrix is greater than $1$ for most cases, 
as a survey over different $\mu\in \left( 0.01,100 \right)$ and $L \in \left\{ 2,\ldots,50 \right\}$ showed.
This renders strategy 2 useless for most of the cases we considered so far.

The short-term convergence on the other hand is not captured by the first 3 strategies.
In contrast, strategy 4 does this very well, as we see in Figure \ref{fig:res_finding_3}.
We also see that the short-term convergence is faster for initial values with a small wave number $k$ and that the long-term convergence speed is almost independent from the initial value. 
Our interpretation is that PFASST is more efficient in reducing the low frequency error modes in space. 
After the first convergence phase, the error consists of a mixture of modes, which is reduced by PFASST likewise, independently from the initial value frequency.

\begin{figure}%[th]
	\centering	
	\subfloat[time-collocation blocks\label{fig:sama_var_predict_strategies}]{\includegraphics[width=0.48\textwidth]{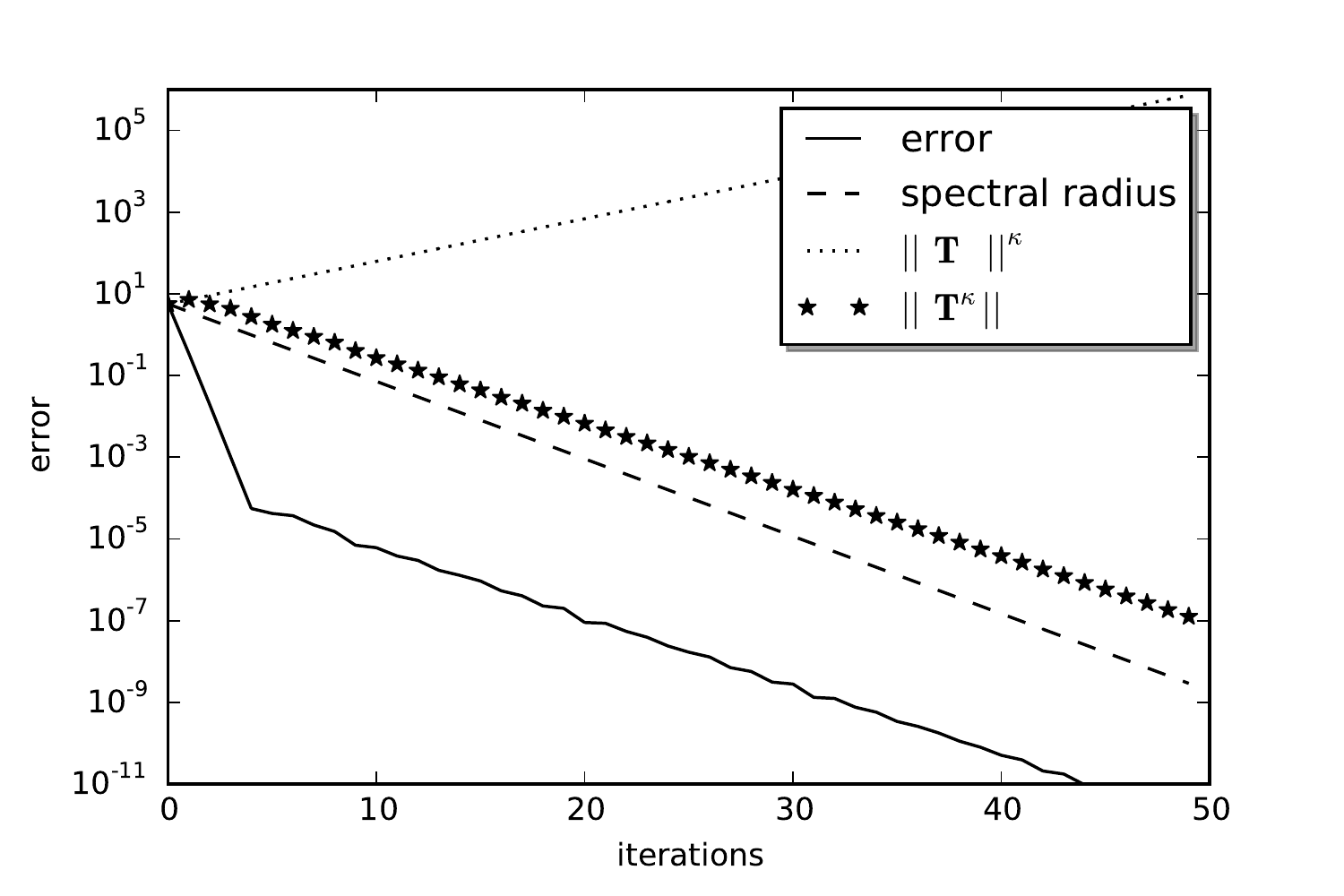}}
	\subfloat[collocation blocks\label{fig:lfa_var_predict_strategies}]{\includegraphics[width=0.48\textwidth]{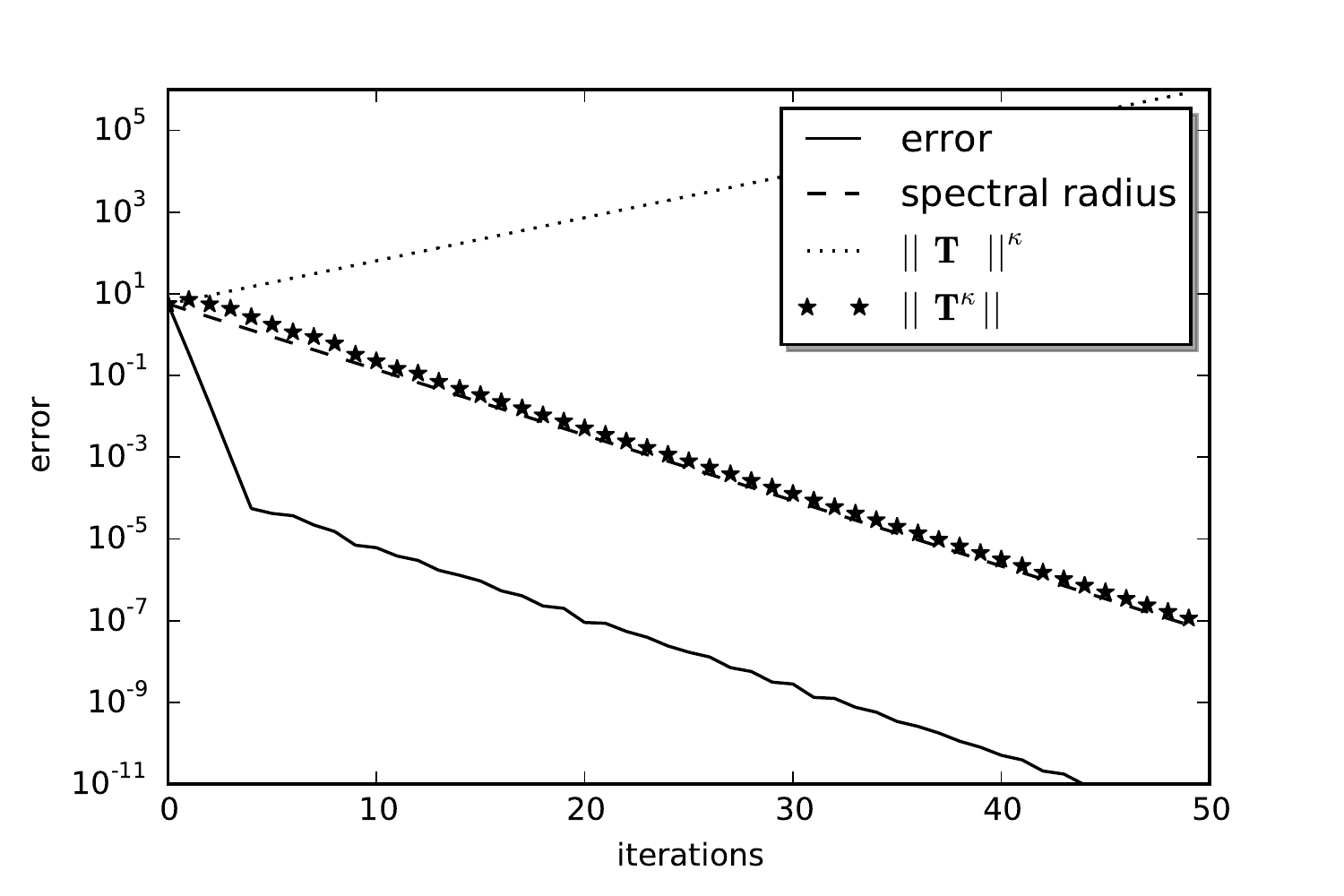}} 
	\caption{
The errors estimates from the strategies 1 to 3, 
compared to the actual error plotted against the number of iterations of PFASST, 
for the diffusion problem with an initial value function of $\sin(2\pi 8 x)$. 
	}
	\label{fig:res_finding_2}
\end{figure}

Here we actually see a difference between the different types of blocks: For the error prediction of the first iterations, using strategy 4 with collocation blocks is not as accurate as using time collocation blocks.
In contrast, no differences in the quality of the error prediction are notable in the long-term convergence phase, again.

\begin{figure}[th]
	\centering	
	\subfloat[\label{fig:sama_strat_4}]{\includegraphics[width=0.48\textwidth]{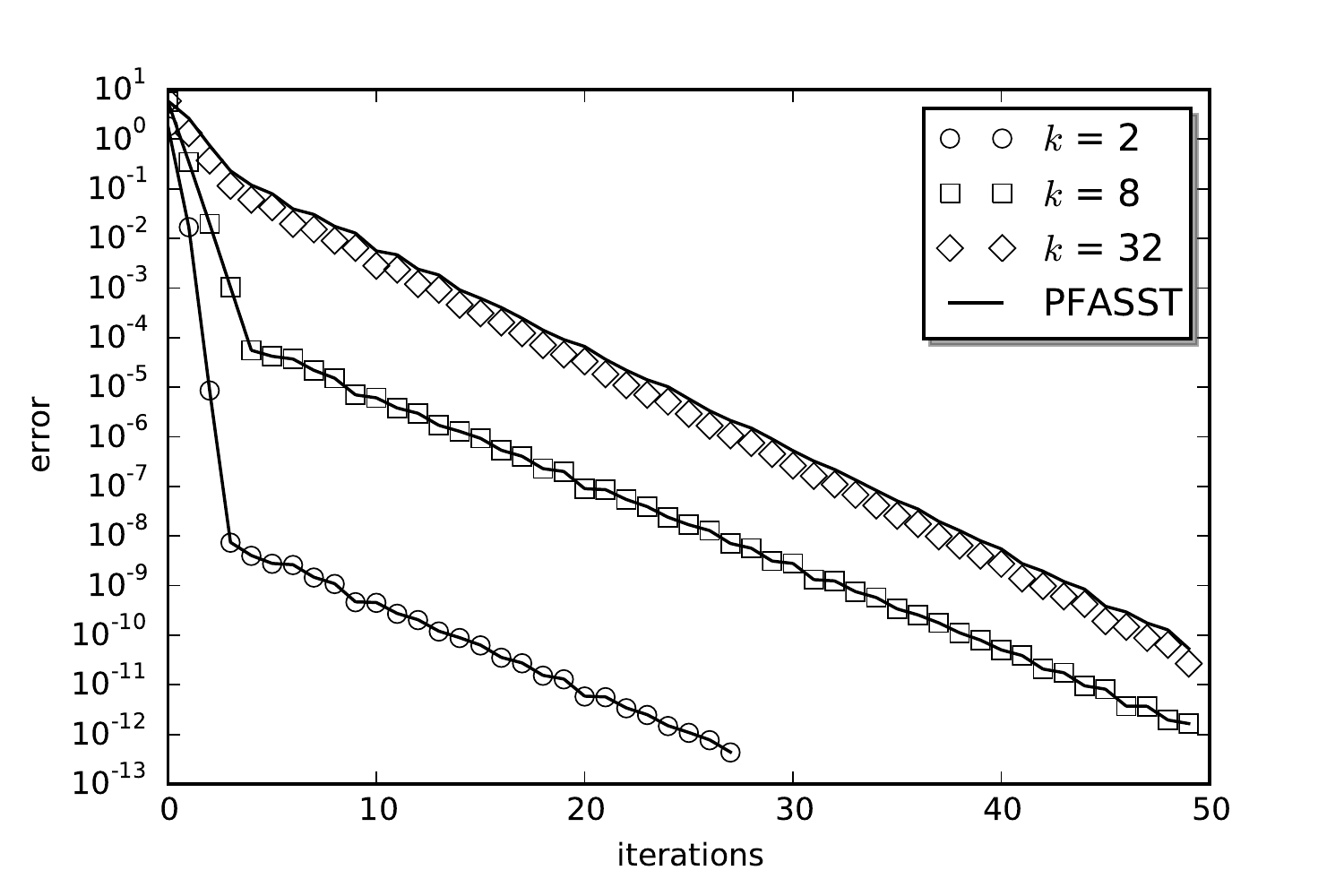}}
	\subfloat[\label{fig:lfa_strat_4}]{\includegraphics[width=0.48\textwidth]{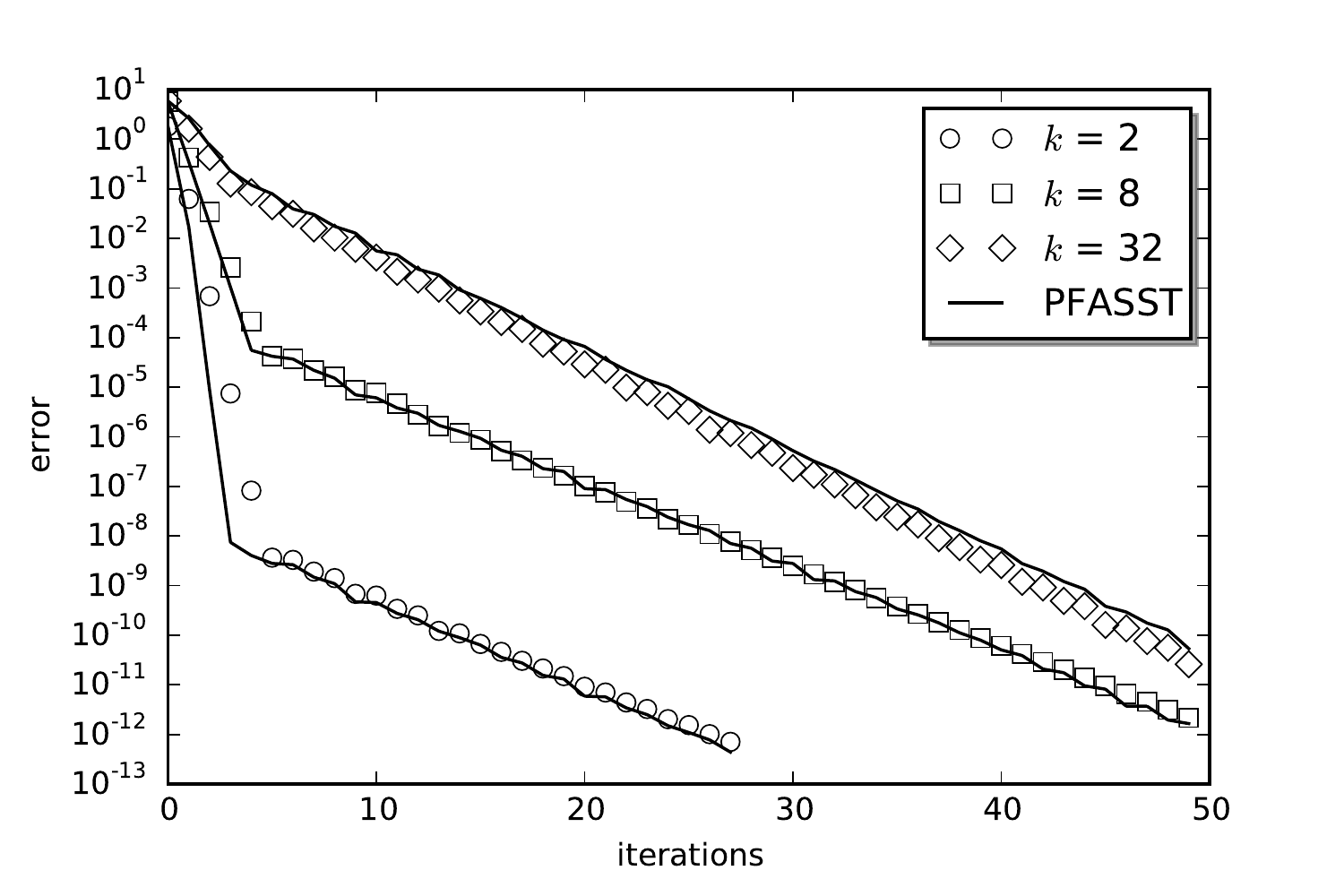}} \newline \centering
	\subfloat[\label{fig:sama_err_diff}]{\includegraphics[width=0.48\textwidth]{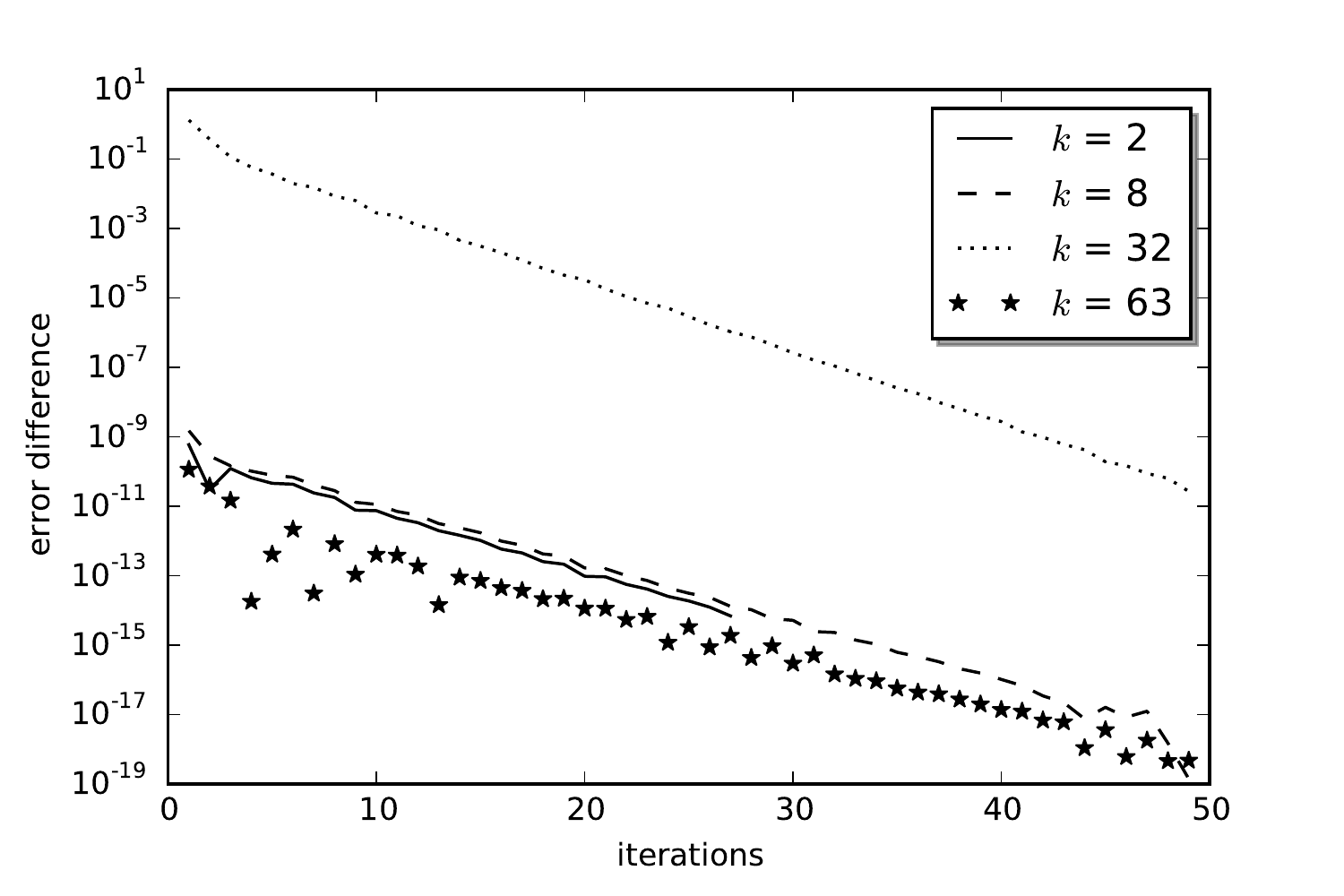}}
	\subfloat[\label{fig:lfa_err_diff}]{\includegraphics[width=0.48\textwidth]{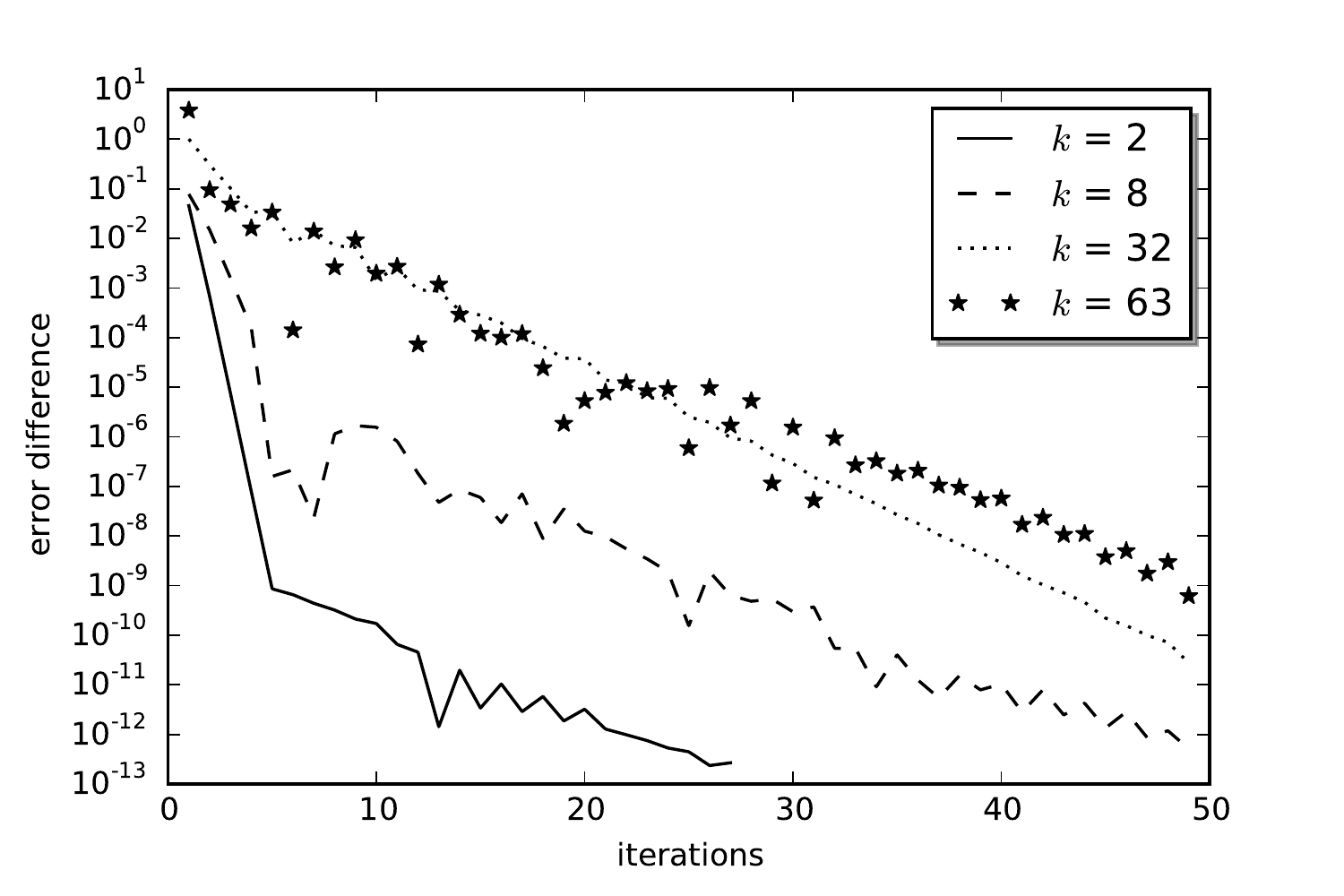}}	
	\caption{Strategy 4 for various $k$, using time collocation blocks on the left and collocation blocks on the right for the diffusion problem.
		 In the second row the difference between the error and the error prediction is plotted.
	}
	\label{fig:res_finding_3}
\end{figure}

\subsection{Advection problem}
\label{ch:advection_problem}

The second prototype problem is the 1D advection equation, given by
\begin{align}
  \begin{split}
  u_t &= c u_x, \quad  x \in [0,1] \mbox{ and } t \in [0,T]\\
  u(x,0) &= u_0(x), \quad u(0,t) = u(1,t) \quad t \in [0,T],
  \end{split}
  \label{eq:adv_eq}
\end{align}
with advection coefficient $c>0$.
The discretization is done in the same manner as~\eqref{eq:discrete_heat_eq}, but
we use an upwind difference stencil of the order $3$, instead of a central difference stencil. 
This yields again a circulant matrix $\matr{F}_D$, 
with eigenvalues and eigenvectors according to \eqref{eq:eig_vals_circ_matr}.
For the numerical experiments we use advection speed $c =4.88 \cdot 10^{-3}$, 
resulting in a CFL number of $62.5 \cdot 10^{-3}$. 
The discretization in space and time is similar to the discretization in the previous section.

\subsubsection{The error vector}

For a initial value function $u_0$ the solution reads
\begin{align*}
  u(t,x) = u_0( x-ct).
\end{align*}
We use again the initial values given by~\eqref{eq:heat_eq_simple_initial_functions}.
The initial values are spread on each node, this yields the initial error vector
\begin{align*}
    \vect{e}^{0} = \left(\vect{e}^{0}_{1}, ..., \vect{e}^{0}_{L}\right)^T
\end{align*}
with
\begin{align*}
  \vect{e}^{0}_{n} &= \left(\vect{e}^{0}_{n,1}, ..., \vect{e}^{0}_{n,M}\right)^T\\
  \vect{e}^{0}_{j,m} &= \left(u_0(x_1) - u_0(x_1-c(t_{j-1}+\tau_m)), ..., u_0(x_N) - u_0(x_N-c(t_{j-1}+\tau_m))\right)
\end{align*}

When the class of initial value~\eqref{eq:heat_eq_simple_initial_functions} is used, 
the initial values can be decomposed again into two modes, belonging to different harmonics.
The analysis is thus again reduced to certain harmonic blocks $\tmatr{B}^{(T)}_{k}, \tmatr{B}^{(T)}_{\frac{N}{2}-k}$
and $\tmatr{B}^{(T)}_{k,j}, \tmatr{B}^{(T)}_{\frac{N}{2}-k,j}$, respectively.
Note, that this computation of the error vector works even if only a numerical solution to the problem is given.

\subsubsection{Error prediction}

For the advection problem the four strategies yield significantly different results than for the diffusion problem.
In Fig.~\ref{fig:adv_strat_1_3} we now observe three phases of convergence: two rapid phases at the beginning and at the end and one almost stagnating phase in the middle. 
We observed these phases for all initial wave numbers $\kappa$, 
with the peculiarity of a decreasing, almost vanishing first phase for increasing $\kappa$.
Regarding the different strategies, we see that only the spectral radii (strategy 1) is able to capture the first phase, 
while the norm of the powers of the iteration matrix (strategy 3) captures the last phase.

\begin{figure}[th]
	\centering	
	\subfloat[time-collocation blocks\label{fig:adv_sama_var_predict_strategies}]{\includegraphics[width=0.48\textwidth]{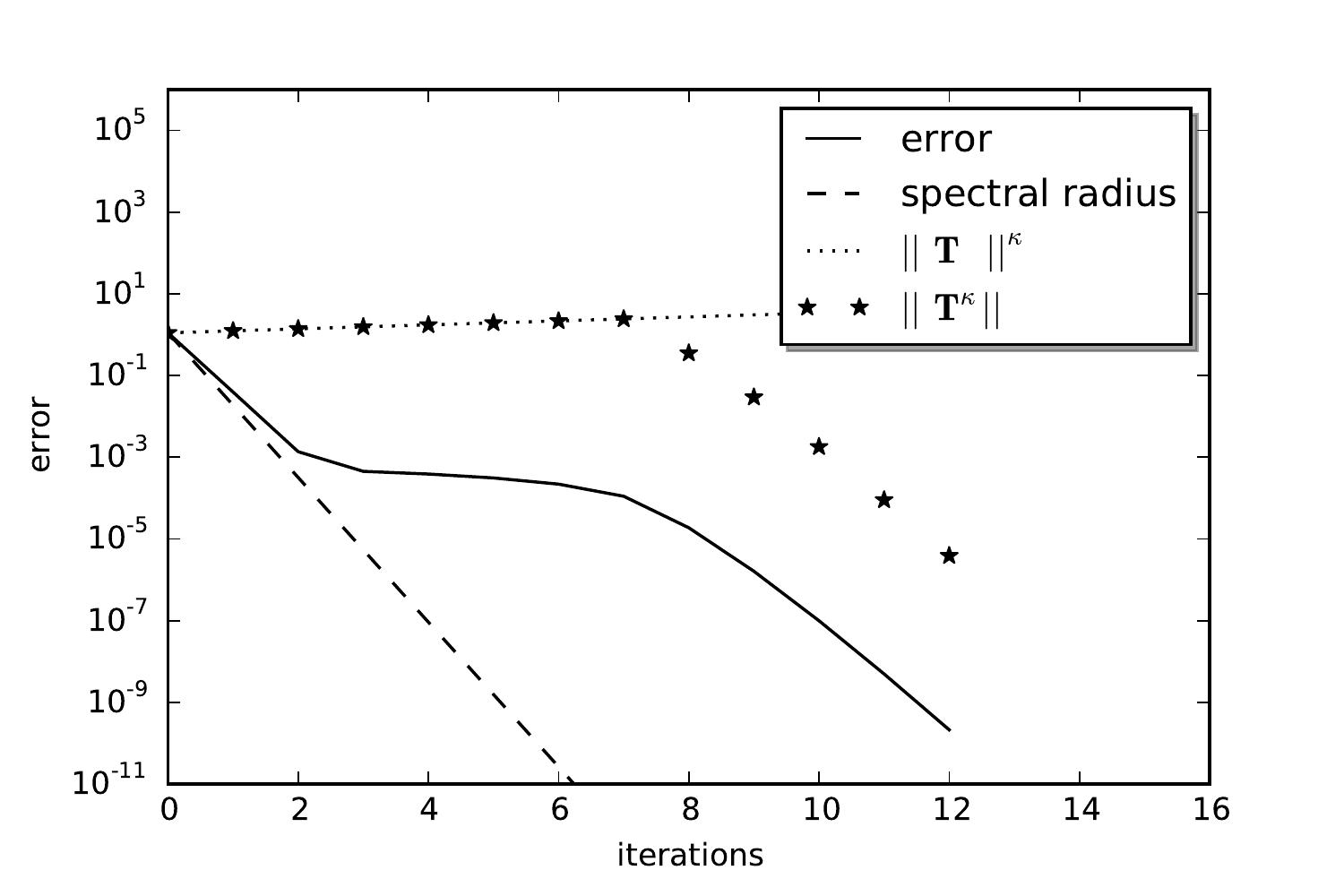}}
	\subfloat[collocation blocks\label{fig:adv_lfa_var_predict_strategies}]{\includegraphics[width=0.48\textwidth]{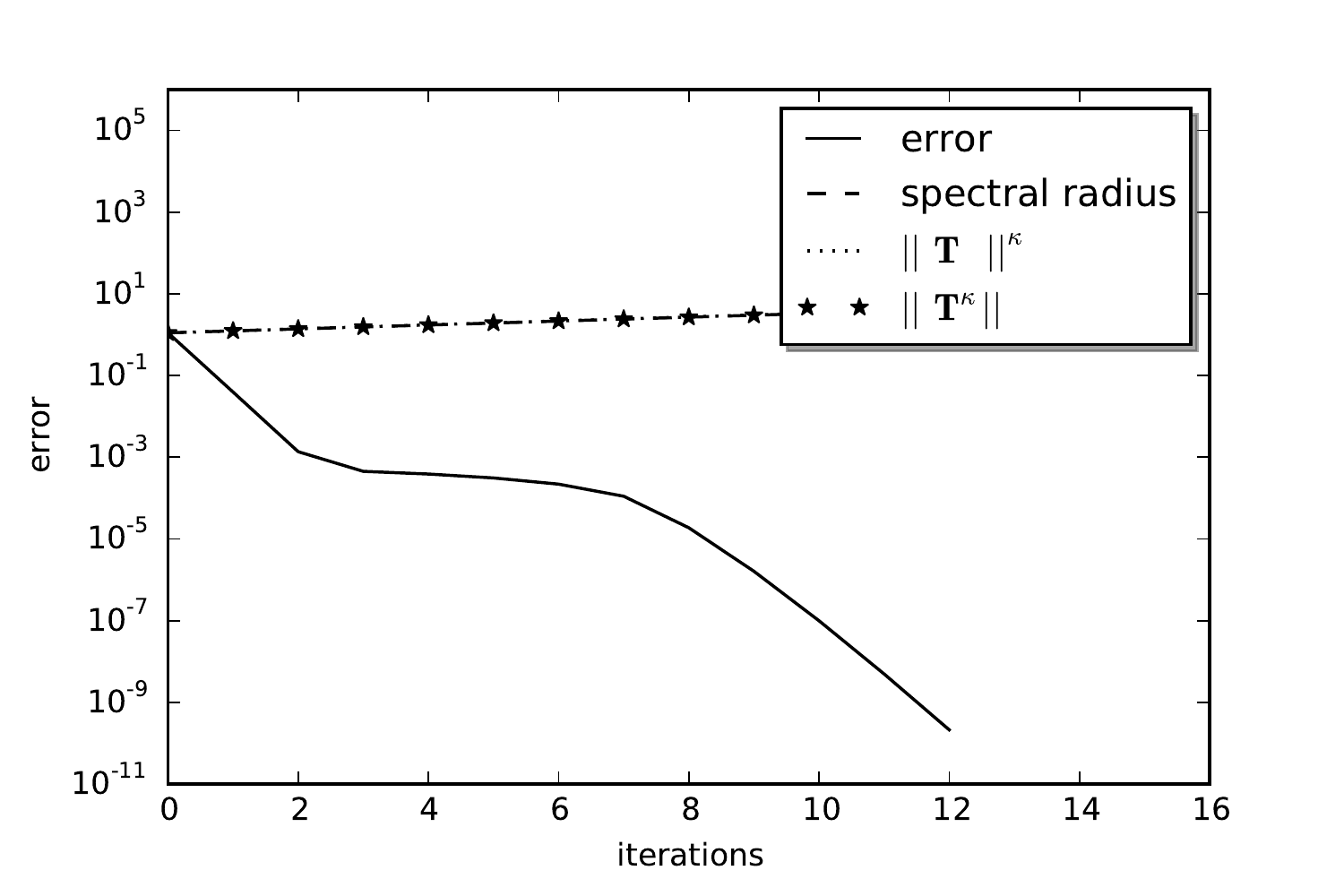}}  	
	\caption{
The errors estimates from the strategies 1 to 3, 
compared to the actual error plotted against the number of iterations of PFASST. 
Advection problem with an initial value function of $u_0(x) = \sin(2\pi 8 x)$. 
	}
	\label{fig:adv_strat_1_3}
\end{figure}

Again strategy 4 is successful in exactly predicting the error, 
when time collocation blocks are used.
On the other hand, for the advection equation the use of collocation blocks only serve as an assessment for initial values with high $k$, and then only for the first phase.
Obviously, the assumption of periodicity in time is not valid for advection-dominated problems. 
Thus, time collocation blocks should be considered in this case.

\begin{figure}[th]
	\centering	
	\subfloat[\label{fig:adv_sama_strat_4}]{\includegraphics[width=0.48\textwidth]{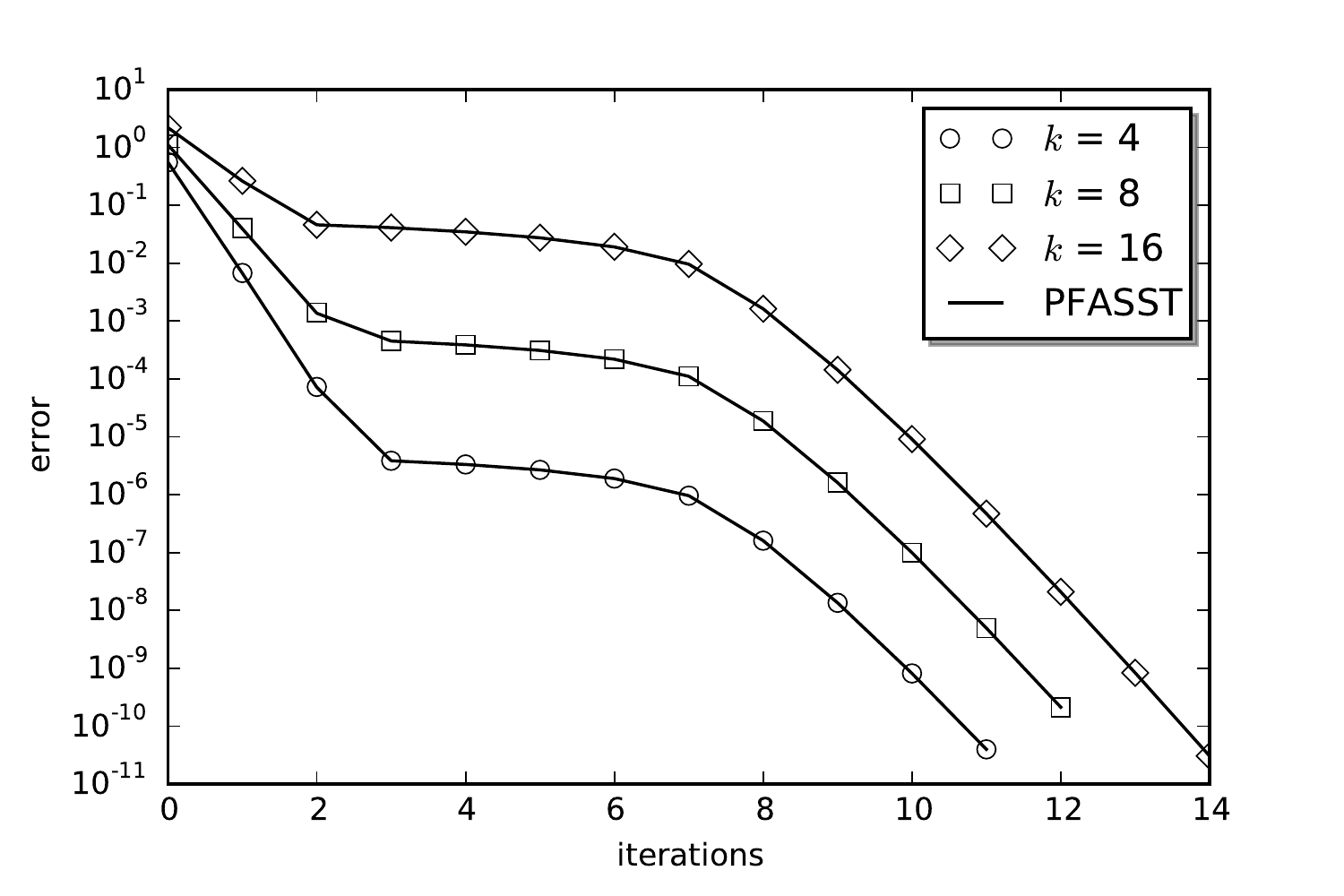}}
	\subfloat[\label{fig:adv_lfa_strat_4}]{\includegraphics[width=0.48\textwidth]{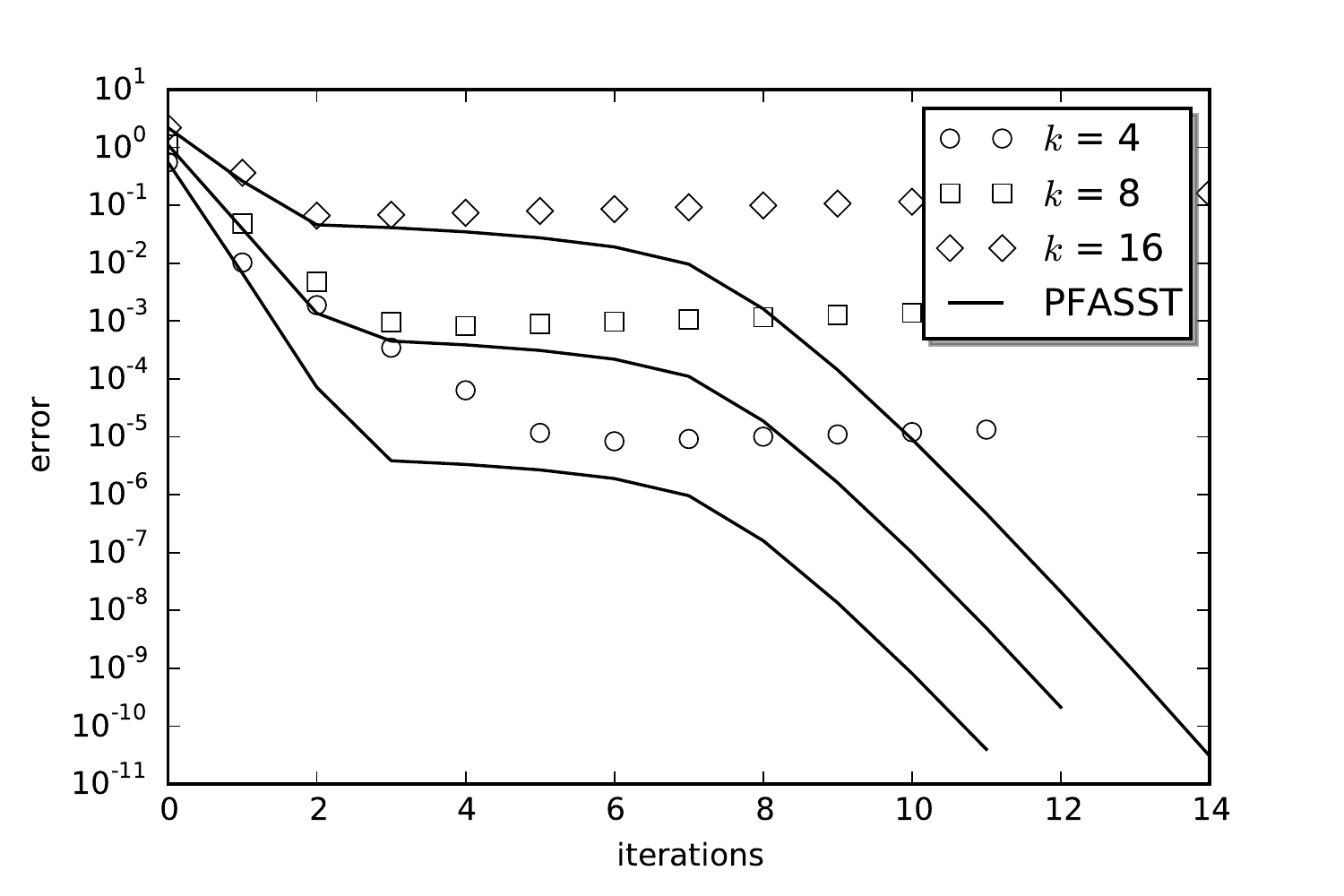}}\newline \centering
	\subfloat[\label{fig:adv_sama_err_diff}]{\includegraphics[width=0.48\textwidth]{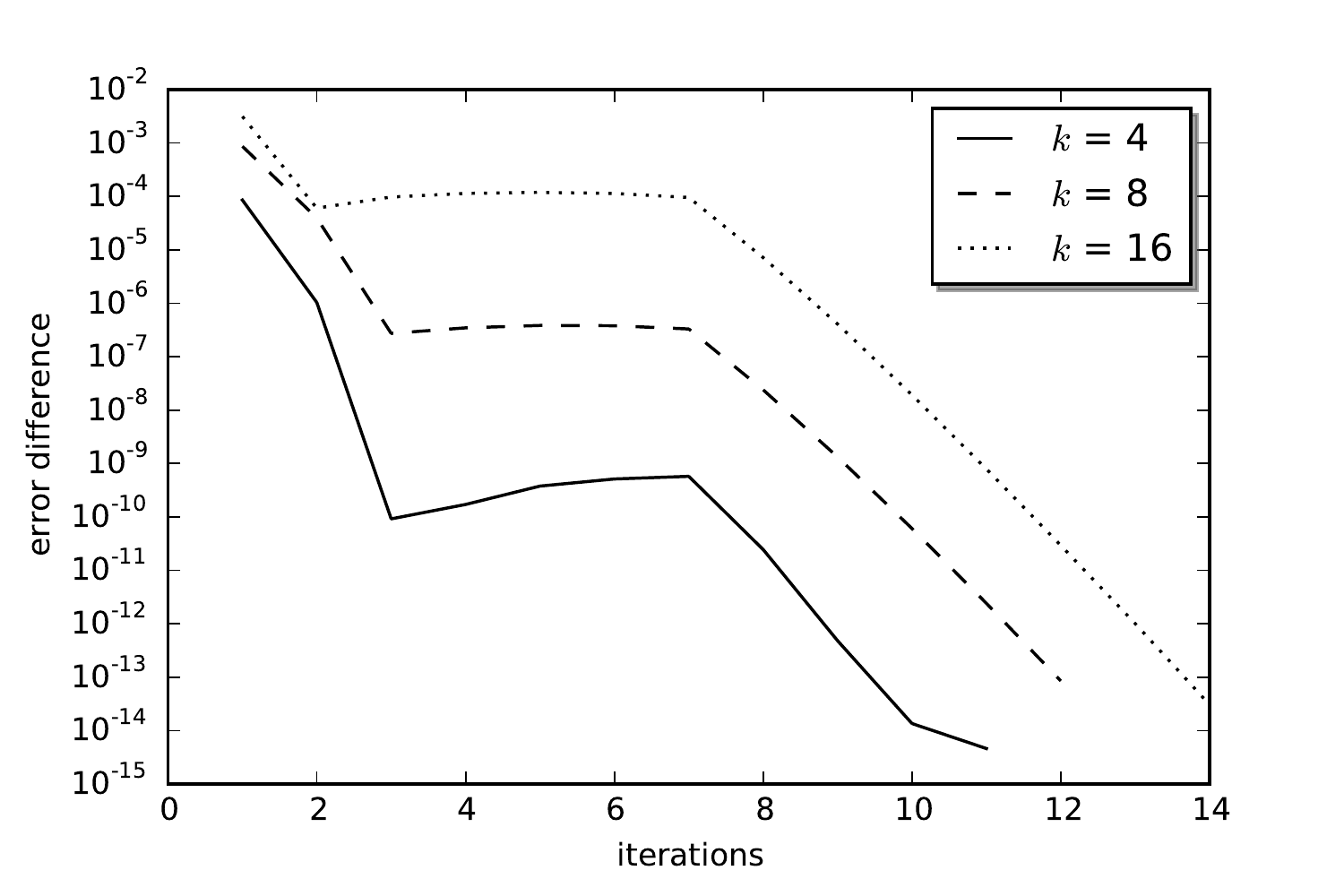}}
	\subfloat[\label{fig:adv_lfa_err_diff}]{\includegraphics[width=0.48\textwidth]{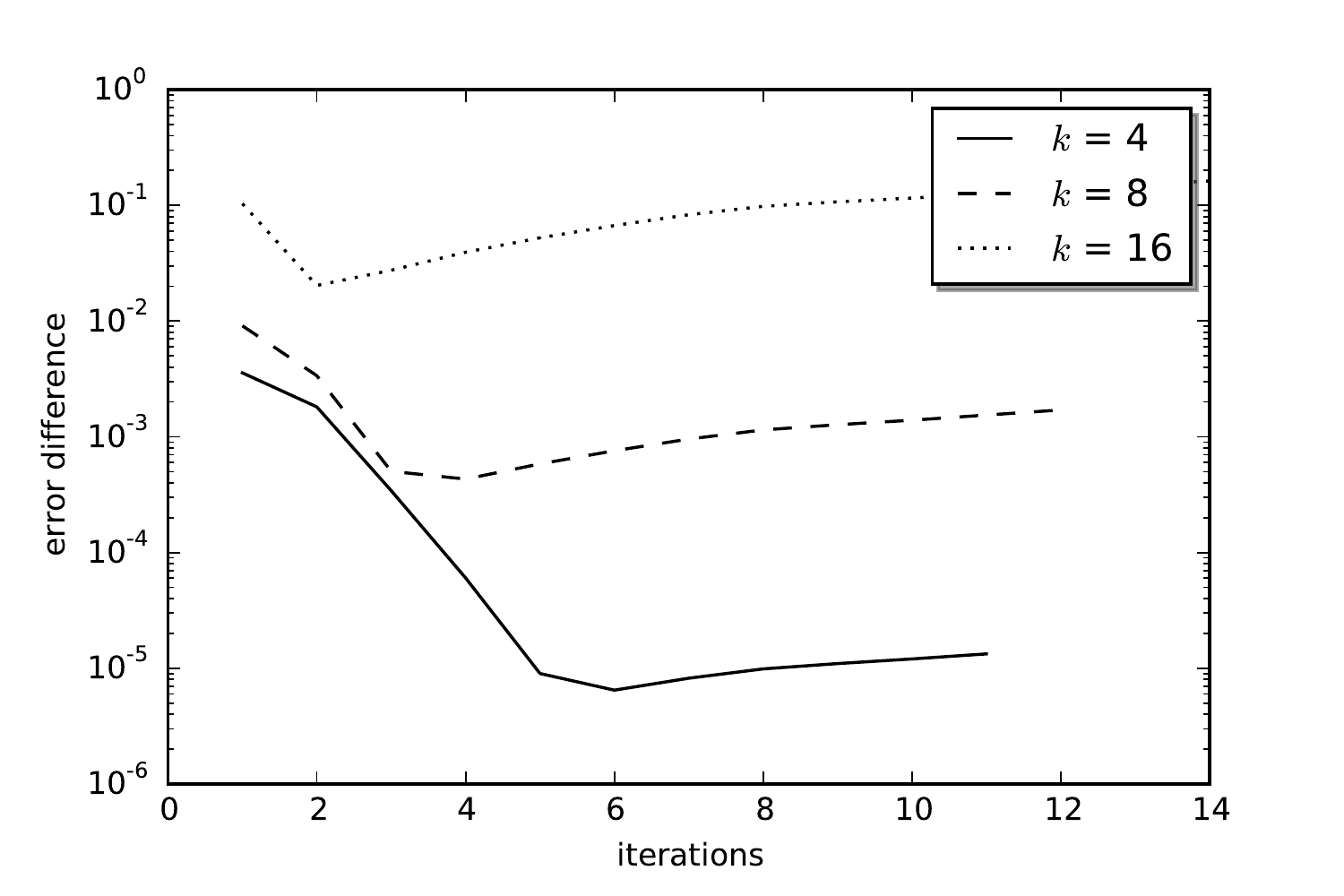}}	
	\caption{Strategy 4 for various $k$, using time collocation blocks on the left and collocation blocks on the right for the advection problem.
		In the second row the difference between the error and the error prediction is plotted.}
	\label{fig:adv_strat_4}
\end{figure}

\section{Conclusion and Outlook}
\label{ch:outro}
In this paper we decomposed the PFASST algorithm into its atomic parts.
Using analogies to classical iterative methods like Gau\ss-Seidel and Jacobi, we described PFASST for two levels and linear problems as a combination of a highly parallel, approximative block Jacobi solver on the fine level and a serial, approximative block Gau\ss-Seidel solver on the coarse level.
With this we could show that for linear problems PFASST is a multigrid algorithm for the composite collocation problem in space and time.
We stated the underlying composite collocation problem in matrix formulation, spanning the full domain in space and time, and decomposed it into three layers: spatial decomposition, time-stepping and quadrature nodes.
With suitable transformations, we could show the similarity of PFASST's iteration matrix to a block-diagonal matrix, containing either $N/2$ time-collocation blocks of size $2ML$ or $NL/2$ collocation blocks of size $2M$. 
While in the first case the analysis is rigorous, in the second case periodicity in time is assumed.

We identified 4 different strategies to test the convergence properties of PFASST using the block diagonalization of the iteration matrix.
Along the lines of two prototype problems, we investigated the quality of the predictions given by these strategies compared to the numerical results form PFASST.
We explored the effect of PFASST on different modes of the solution, 
depending on the initial values.

%We found two phases of convergence, which occurred in various forms depending on the mode. 
%What we could show was that from a certain dispersion numbers $\mu$ on, 
%which controls the numerical nature of the problem, 
%it is possible, with assumption in \ref{ch:periodicity_in_time}, 
%to predict the second convergence phase as accurate as the semi algebraic approach does it. 
%We also have seen that the first phase isn't predictable without knowing the error a priori.

%Further distant goals are a further development of the analysis,
%so that the smoothing and approximation property of PFASST will be examined at least semi algebraically, 
%a comparison in different aspects to other multigrid algorithms,
%and a investigation of PFASST as a preconditioner for iterative solvers like GMRES and Krylov subspace methods.

With a suitable measure for the convergence speed of PFASST at hand, the central next step would be to estimate the parallel performance of this algorithm in comparison to serial SDC runs. 
To this end, block diagonalizations of PFASST and SDC can be compared following the strategies presented in this work. 
This would augment the current speedup considerations of PFASST as stated in~\cite{EmmettMinion2012} by providing estimates for the actual iterations counts.
In addition, we have identified the following topics as relevant for further studies.

\paragraph{Detailed parameter and component studies.}
So far, we have only investigated simple 1D problems, demonstrating how the LFA of the iteration matrix can be used to predict the convergence behavior of PFASST for different situations.
These examples can serve as a blueprint for a much deeper and more detailed analysis of PFASST's convergence properties for various problems.
Also, the matrix formulation of PFASST allows us to exchange parts more easily. 
We can test other smoothers than SDC, change the quadrature rules used on the subintervals, vary interpolation and restriction on space (and even time) and apply iterative solvers like standard multigrid in space for inverting the spatial operators.

%\TODO{OK, now it's your turn. See comments in the source for inspiration. - Tried my best . . .}
\paragraph{Non-linear functions.}
% restricted to linear problems in order to apply LFA
% already used FAS formulation for the algorithm
% extension to non-linear functions/right-hand sides is straightforward
% but then: convergence analysis unclear
We restricted our self to linear problems in order to apply the Local Fourier Analysis. 
However, the notation used is derived from the Full Approximation Scheme 
and therefore is also applicable to non-linear right-hand sides.
Meaning that we use a non-linear function $ f\left( \tvect{U}\left( \tau \right) \right) $ of $\matr{A}\tvect{U}\left( \tau \right)$, 
also meaning that most matrices are exchanged by operators. 
These changes make a convergence analysis more difficult.

\paragraph{Extension to multiple levels.}
% standard PFASST algorithm is designed to use more than two levels (in contrast to Parareal)
% did not exploit the fact that only two levels were used, only for simplicity of the notation
% difference to standard MG theory: no exact solution on the coarse level here, so interesting to see what happens if multiple levels are used
% also: formalism derived in this paper is open to more advanced interpolation and restriction operators: also coarsening in time and quadrature possible
% towards a full space-time multigrid
In contrast to Parareal, the PFASST algorithm is designed to use more than two levels. 
Due to the simplification of the notation and the rigor of the argumentation chain, 
this fact was not exploited. 
For the same reasons, the interpolation and restriction matrices effected only the spatial dimension, 
although it is possible to construct coarse levels with less quadrature nodes than on the fine level. 
The effects on the formalism in Section~\ref{ch:lfa_and_transformation_matrices} would be minor. 
It is another story, if a coarse level is constructed where two or more subintervals from the fine level are merged to one. 
This would be a step in the direction of full Space-Time MultiGrid, but some work is needed to adjust the formalism in Section~\ref{ch:lfa_and_transformation_matrices} 
for a similar convergence analysis.
Another step towards full ST-MG would be the use the exact solution on the coarsest level instead of one or more SDC sweeps,
but first brief experiments showed no significant difference between the use of the exact solution or the use of SDC Sweeps.

\paragraph{Rigorous convergence analysis.}
% so far, we used the LFA to estimate the convergence of PFASST
% one of the standard techniques for a more rigorous convergence analysis is the analysis of smoothing and approximation property
% major obstacle here: only numerical treatment of the collocation layer possible, no matrix properties of the Q and Qd-matrices known
The usual attempt in Multigrid theory for a rigorous convergence analysis contains the proof of the smoothing and approximation property.
Both endeavors are difficult on their own, but, in our case, are further impeded by the matrices $\matr{Q},\matr{Q}_{\Delta}$. 
These matrices are dense and yield no obvious structural properties, which could be exploited. 
First steps towards a more rigorous analysis would be to resolve this problem.

%\TODO{You have to cleanup the references..: remove all URLs, check for capital letters (e.g. ``rungekutta''), missing bib data (e.g. Ref. 38), updated refs (e.g Ref. 35), ...}

\appendix
%\input{appendix}
%\newpage
%This page is intentionally left blank
\bibliographystyle{wileyj}
\bibliography{books,sdc,Pint,Pint_Self}

\end{document}